\documentclass[11pt]{amsart}

\usepackage{amsfonts,amsmath,amssymb,amsthm}

\usepackage[all]{xy}
\xyoption{matrix}
\usepackage[svgnames,hyperref]{xcolor}
  
\usepackage{mathabx} 
\newcommand{\nicecolor}{Navy}  
\usepackage[shortlabels]{enumitem}
\setlist[1]{wide}
\setlist[2]{leftmargin=15mm} 
\setlist[enumerate]{label=\rm{(\arabic*)}}
\setlist[enumerate,2]{label=\rm({\it\roman*}), }
\setlist[itemize]{label=\raisebox{0.25ex}{\tiny$\bullet$}}
\usepackage[backref, colorlinks, linktocpage, citecolor = \nicecolor, linkcolor = \nicecolor]{hyperref} 
\usepackage{mathtools}
\usepackage{soul} 

\newtheorem{mytheorem}{Theorem}

\newtheorem{lemma}{Lemma}[section]
\newtheorem{corollary}[lemma]{Corollary}
\newtheorem{theorem}[lemma]{Theorem}
\newtheorem{proposition}[lemma]{Proposition}

\theoremstyle{definition}
\newtheorem{definition}[lemma]{Definition}
\newtheorem{notation}[lemma]{Notation}

\theoremstyle{remark}
\newtheorem{remark}[lemma]{Remark}

\newtheorem{example}[lemma]{Example}

\newcommand\C{{\mathbb C}}
\newcommand\GL{{\mathrm{GL}}}

\newcommand\R{{\mathbb R}}
\newcommand\A{{\mathbb A}}
\newcommand\Z{{\mathbb Z}}

\newcommand\iso{\stackrel{\simeq}{\longrightarrow}}
\newcommand{\Aut}{\mathrm{Aut}}
\newcommand{\Bir}{\mathrm{Bir}}
\newcommand{\End}{\mathrm{End}}
\newcommand{\Jac}{\mathrm{Jac}}

\newcommand{\Spec}{\mathrm{Spec}}
\newcommand{\BA}{\mathrm{BA}}
\newcommand{\Aff}{\mathrm{Aff}}
\newcommand{\Gal}{\mathrm{Gal}}
\newcommand{\id}{\mathrm{id}}

\newcommand{\KK}{\mathbf{K}}
\newcommand{\kk}{\mathbf{k}}
\newcommand{\I}{\mathbf{i}}
\newcommand{\e}{\mathrm{e}}

\title[Gizatullin surfaces and Koras-Russell threefolds]{Real forms of some Gizatullin surfaces and Koras-Russell threefolds}
\author[J.~Blanc]{J\'er\'emy~Blanc}
\address{J\'er\'emy Blanc, Universit\"{a}t Basel, Departement Mathematik und Informatik, Spiegelgasse $1$, CH-$4051$ Basel, Switzerland.}
\email{jeremy.blanc@unibas.ch}
\urladdr{http://algebra.dmi.unibas.ch/blanc}

\author[A.~Bot]{Anna~Bot}
\address{Anna Bot, Universit\"{a}t Basel, Departement Mathematik und Informatik, Spiegelgasse $1$, CH-$4051$ Basel, Switzerland.}
\email{annakatharina.bot@unibas.ch}
\urladdr{http://algebra.dmi.unibas.ch/bot}

\author[P.-M.~Poloni]{Pierre-Marie~Poloni}
\address{Pierre-Marie~Poloni, Universit\"{a}t Basel, Departement Mathematik und Informatik, Spiegelgasse $1$, CH-$4051$ Basel, Switzerland.}
\email{ poloni.pierremarie@gmail.com}
\urladdr{http://algebra.dmi.unibas.ch/poloni}

\thanks{The first two authors acknowledge support by the Swiss National Science Foundation Grant \textquotedblleft Geometrically ruled surfaces\textquotedblright 200020--192217.}
\begin{document}


\begin{abstract}
We describe the real forms of Gizatullin surfaces of the form $xy=p(z)$ and of Koras-Russell threefolds of the first kind. The former admit zero, two, three, four or six isomorphism classes of real forms, depending on the degree and the symmetries of the polynomial~$p$. The latter, which are threefolds given by an equation of the form $x^dy+z^k+x+t^\ell=0$, all admit exactly one real form up to isomorphism. 
\end{abstract}

\maketitle 
\tableofcontents

\section{Introduction} 

Given a complex algebraic variety $X$, a \textit{real form} of $X$ is a real algebraic variety $Y$ whose complexification is isomorphic to $X$. It is then natural to ask whether $X$ has  one, only one, finitely many or infinitely many isomorphism classes of real forms. We study here the case where $X$ is affine. The most natural examples to look at in this context are the affine spaces. For any $n \geq 1$, an obvious real form of $\A^n_\C$ is $\A^n_\R$. For $n\le 2$, it turns out to be the only one up to isomorphisms. This is a nice exercise for $n=1$, and for $n=2$ it is a result of Kambayashi in \cite[Theorem 3]{Kambayashi} based on the amalgamated free product structure of $\Aut(\A^2_\C)$. For $n\ge 3$, it is still unknown whether $\A^n_\C$ admits any nontrivial real form.

In this article, we investigate some affine surfaces and threefolds which are close to the affine plane and space.

Recall that a \emph{Gizatullin surface}  is a normal complex affine surface completable by a \emph{zigzag}, that is, by a simple normal crossing divisor with rational components and a linear dual graph, see for more details \cite{FKZ}.  These surfaces are classical generalisations of the affine plane. For instance,  a smooth affine surface is quasi-homogeneous (that is, its automorphism group admits an open orbit with finite complement) if and only if it is a Gizatullin surface or isomorphic to $(\A^1_\C\setminus \{0\})^2$, see \cite{GizatullinQH}. Moreover, by \cite[Theorem]{DuboulozML}, a normal complex affine surface admits two $(\C,+)$-actions with different general fibres if and only if it is a Gizatullin surface not isomorphic to $\A^1_\C\times (\A^1_\C\setminus \{0\})$. In the latter case, the zigzag can be chosen to have a sequence of self-intersections $(0,-1,-a_1,\ldots,-a_r)$, with $a_1,\ldots,a_r\ge 2$ (see for instance \cite{BD}). 

The case   $r=0$ is the affine plane $\A^2_\C$. The case  $r=1$ corresponds to the surfaces  $D_p=\Spec(\C[x,y,z]/(xy-p(z)))$, where $p\in \C[z]$ is of degree at least~$2$, called \emph{Danielewski surfaces} by some authors. For $r=2$, there are Gizatullin surfaces with uncountably many nonisomorphic real forms, as the second author recently proved in \cite{Bot}. 
In this text, we compute the number of isomorphism classes of real forms of all surfaces $D_p$, and show in particular that this number is finite for all of them. 

We first establish in Proposition~\ref{Prop:ExistencerealstructureDanielewski} that  $D_p$ admits a real form if and only if there exist $a,\lambda \in \C^*$, $b\in \C$, such that $\lambda p(az+b)\in \R[z]$. In this case, we can assume that $p\in \R[z]$, and moreover that $p$ is in \emph{reduced} form as defined in Definition~\ref{Defi:ReducedForm}, i.e., that $p(z)=z^d+s(z)$ for some integer $d$ and some polynomial $s\in\R[z]$ with $\deg(s)\leq d-2$. We then obtain the full list of isomorphism classes of real forms for any such surface in Propositions~\ref{Prop:RealformsDanielewski2a},~\ref{Prop:RealformsDanielewski2b} and~\ref{Prop:RealformsDanielewski3}, summarised as follows:

\begin{mytheorem}\label{ThmA}
Let $p\in \R[z]$ be a polynomial of degree $d\ge 2$ in reduced form. Write $p(z)=z^mq(z^n)$ where $m\ge 0$, $n\ge 1$, $q\in \R[z]$, $q(0)\neq 0$ and where $q,n$ are chosen such that $n$ is  maximal if $q\neq 1$. For all $a,b,c\in \{0,1\}$, the surface
\[S_{abc}=\Spec(\R[x,y,z]/(x^2+(-1)^a y^2+(-1)^b z^mq((-1)^c z^n)))\]
is a real form of the Gizatullin surface $D_p=\Spec(\C[x,y,z]/(xy-p(z)))$. Moreover, the number~$i$ of isomorphism classes of real forms of $D_p$ and the representatives are related as follows.
\[\begin{array}{|l|l|l|l|}
\hline
i & \text{Representatives} & \multicolumn{2}{c|}{\text{Conditions on }q,n,d}\\
\hline
2& S_{000}, S_{110} & q=1,d=2 & q=1,d\ge 3\text{ odd}\\ \hline
3 &S_{000},S_{010},S_{110}& q=1,d\ge 4\text{ even}&q\neq 1, n\text{ odd}\\ \hline
4 &S_{abb},a,b\in \{0,1\} & q\neq 1, n\text{ even}, d\text{ odd}& q\neq 1, (n,d)=(2,2)
 \\ \hline
6 &S_{00c},S_{a1c},a,c\in \{0,1\}& \multicolumn{2}{l|}{ q\neq 1, n,d \text{ both even}, (n,d)\neq (2,2)}\\
\hline
\end{array}\]
\end{mytheorem}
 
Just as for the affine plane, the automorphism group of a surface  $D_p$  has the structure of a free product of two subgroups  amalgamated over their intersection (see Theorem~\ref{thm: BD11, autos of A1-fibred affine surfaces} below, or \cite[Theorem~5.4.5]{BD}). The situation is however more complicated than for $\A^2_\C$, since the cohomology pointed sets of the two factors are not trivial.

In the particular case of the affine quadric $\Spec(\R[x,y,z]/(xy-z^2+1))$, Theorem~\ref{ThmA} provides exactly four isomorphism classes of real forms, given by  $\Spec(\R[x,y,z]/(x^2\pm  y^2+ z^2\pm 1))$. This rectifies a similar claim in the introduction of \cite{DFM}, where only three of the four real forms were given. 

To complete our study of real forms of affine surfaces, we consider in Section~\ref{CCC} the surfaces $(\A^1_\C\setminus \{0\})^2$ and $\A^1_\C\times (\A^1_\C\setminus \{0\})$ mentioned above. We prove that they admit six and four isomorphism classes of real forms, respectively.

Following the examination in dimension two, we move to the study of three-dimensional affine varieties. We investigate the Koras-Russell threefolds of the first kind in Section~\ref{Section:KR}. We recall that they are defined as the hypersurfaces 
\[X_{d,k,\ell}=\{x^dy+z^k+x+t^\ell=0\}\subset\A^4_{\C},\]
where $d\geq2$ and $2\leq k<\ell$ are integers with $k$ and $\ell$ relatively prime, and that they are all smooth affine contractible, hence diffeomorphic to $\R^6$ when equipped with the euclidean topology \cite{Dimca}. They are furthermore $\A^1_{\C}$-contractible in the $\A_{\C}^1$-homotopy sense \cite{Dubouloz-Fasel}. Nevertheless,  none of them is  isomorphic to $\A^3_{\C}$ as an algebraic variety \cite{ML-cubique, Kaliman-ML}. We also recall that two important questions about them are still wild open   for all $d,k,\ell$: It is not known whether $X_{d,k,\ell}$ is biholomorphic to $\A^3_{\C}$, nor whether its cylinder $X_{d,k,\ell}\times\A^1_{\C}$ is isomorphic to $\A^4_{\C}$ (algebraically or analytically).

We prove in Section~\ref{section:KR} that every Koras-Russell threefold of the first kind admits no nontrivial real forms. 
\begin{mytheorem}\label{ThmB}
For all integers $d,k,\ell$ with $d\geq2$ and $2\leq k<\ell$  with $k$ and $\ell$ relatively prime, every real form of the Koras-Russell threefold \[X_{d,k,\ell}= \Spec(\C[x,y,z,t]/(x^dy+z^k+x+t^\ell))\] is isomorphic to the real surface $\Spec(\R[x,y,z,t]/(x^dy+z^k+x+t^\ell))$.
\end{mytheorem}

To  achieve this result, we  use the  structure of the automorphism group of the threefold $X_{d,k,\ell}$ as a subnormal series as computed in \cite{DMJP, Moser-Jauslin} (see Proposition~\ref{prop:Automorphisms_KR}). The factor groups being isomorphic to $\C^*$, $(\C[x,z],+)$ or $\{f\in\Aut_{\C[x]}(\C[x,z,t])\mid f\equiv\mathrm{id}\mod (x^d)\}$,  the key step in the proof of Theorem~\ref{ThmB} is then to  show that the first cohomology pointed set of this latter group is trivial  for any $d\geq 0$. Note that the triviality of this group for $d=0$ also implies that every real structure of $\A^3_\C$ compatible with the projection along one coordinate is equivalent to the standard real structure of $\A^3_\C$, see Proposition~\ref{AutR}.

\subsection*{Acknowledgements.} The authors thank Ronan Terpereau for interesting discussions related to real forms. 

\section{Notation, definitions and reminders}

\subsection{Polynomial maps and variables}

\begin{notation}
Let $n\geq1$ be an integer and $R$ be a commutative algebra over a field $\kk$. We denote by $\End(\A^n_{R})=\End_{R}(\A^n_{R})$ the monoid of  algebraic endomorphisms of $\A^n_{R}=\A^n_{\kk}\times_{\Spec(\kk)}\Spec(R)$. These are the morphisms
of the form \[f\colon (x_1,\ldots,x_n)\mapsto(f_1(x_1,\ldots,x_n),\ldots,f_n(x_1,\ldots,x_n)),\]
where $f_1,\ldots,f_n\in R[x_1,\ldots,x_n]$. As usual, we shall denote such a morphism simply by $f=(f_1,\ldots,f_n)$ and often replace the variables $x_1,x_2,x_3$ by $x,y,z$ if $n\le 3$. 

Given $f=(f_1,\ldots,f_n)\in\End(\A^n_{R})$, we denote by $f^*$ the corresponding $R$-algebra endomorphism of $R[x_1,\ldots,x_n]$ defined by $f^*(P)=P(f_1,\ldots,f_n)$   for all $P\in R[x_1,\ldots,x_n]$. In particular, $f^*(x_i)=f_i$ for $i=1,\ldots,n$. 
\end{notation}

\begin{notation}
We denote by $\Aut(\A^n_R)=\Aut_{R}(\A^n_R)$ the group of algebraic automorphisms of $\A^n_R$ over $\Spec(R)$, by 
\[\Aff_n(R)=\left\{f\in\Aut(\A^n_R)\mid \deg(f^*(x_i))=1 \textrm{ for all } 1\leq i\leq n\right\}\]
the subgroup of affine automorphisms and by 
\[\BA_n(R)=\left\{f\in\Aut(\A^n_R)\mid f^*(x_i)\in R[x_1,\ldots,x_i] \text{ for all } 1\leq i\leq n \right\}\] 
the subgroup of triangular automorphisms.
\end{notation}

We recall that, in dimension two, affine and triangular automorphisms generate all automorphisms of $\A^2_{\kk}$ for any field $\kk$. Moreover, $\Aut(\A^2_{\kk})$ has then the structure of an amalgamated product.

\begin{theorem}[Jung--van der Kulk Theorem]\cite[Theorem 2]{Kambayashi}\label{JungvdK}
Let $\kk$ be a field. Then, the group $\Aut(\A^2_{\kk})$ is the free product 
\[\Aut(\A^2_{\kk})=\Aff_2(\kk)\Asterisk_{\cap}\BA_2(\kk)\] of its affine and triangular subgroups
amalgamated over their intersection.
\end{theorem}

\begin{notation} 
We denote by  
\[\Jac(f)=\begin{vmatrix} \frac{\partial f_1}{x_1}& \cdots & \frac{\partial f_1}{x_n}\\
\vdots & \ddots& \vdots \\
\frac{\partial f_n}{x_1}& \cdots & \frac{\partial f_n}{x_n}\end{vmatrix}\in R[x_1,\ldots,x_n]\] 
the determinant of the Jacobian matrix of any $f=(f_1,\ldots,f_n)\in\End(\A^n_{R})$. We recall that $\Jac(f)\in R^{\times}$ if $f\in\Aut(\A^n_R)$.
\end{notation}

\begin{definition}
A polynomial $P\in R[x_1,\ldots,x_n]$ is called a \textit{variable} if there exists an automorphism $f$ in $\Aut(\A^n_R)$ such that $f^*(x_1)=P$.
\end{definition}

\subsection{Group cohomology, real structures and real forms}

\begin{definition}\label{def:cocycles}
For each group $(G,\circ)$ on which $\mathrm{Gal}(\C/\R)$ acts, we denote by $\alpha\mapsto \overline{\alpha}$ the action of the nontrivial element of $\mathrm{Gal}(\C/\R)$ and by $Z^1(G)\coloneqq Z^1(\mathrm{Gal}(\C/\R),G)=\{\nu\in G\mid \nu\circ \overline{\nu}=1\}$ the set of \textit{$1$-cocycles}. We say that two $1$-cocycles $\nu,\tau$ are \textit{equivalent} if there exists $\alpha\in G$ such that $\tau=\alpha^{-1}\circ\nu\circ \overline{\alpha}$. The cohomology set $H^1(G)\coloneqq H^1(\mathrm{Gal}(\C/\R),G)$ is the set of equivalence classes of $1$-cocycles. It is a pointed set, with a distinguished trivial element, denoted by $1$, which is the class of the identity.
\end{definition}

Since we will need them later in the text, we collect here the cohomology sets of some classical groups.

\begin{lemma}\label{Lem:H1CCstar} Let $n\geq1$ be an integer. We consider the standard action of $\mathrm{Gal}(\C/\R)$ on $\C^n$, on polynomials and on matrices via the complex conjugation of their coefficients.
 \begin{enumerate}[leftmargin=*]
\item\label{H1CCstar} The cohomology pointed sets 
\[H^1(\C^n), H^1(\C^*), H^1(\C[x_1,\ldots,x_n])\]
of the groups $(\C^n,+),(\C^*,\cdot), (\C[x_1,\ldots,x_n],+)$ are trivial.
\item\label{Mun} Let $\mu_n=\{c\in\C\mid c^n=1\}$ be the group of $n$-th roots of unity. The cohomology set $H^1(\mu_n)$ is trivial if $n$ is odd and contains two elements if $n$ is even. These two elements are the class of squares, that is, the class of $1$, and the class of non-square elements, namely, the class of any generator of $\mu_n$.
\item\label{PGL2C}
The cohomology set $H^1(\mathrm{PGL}_2(\C))$ contains exactly two elements. The first one is the set of classes of all matrices $A\in \mathrm{SL}_2(\C)$ with $A\cdot \overline{A}=\left(\begin{smallmatrix} 1 & 0 \\ 0 & 1 \end{smallmatrix}\right)$. The second one is the set of classes of all $A\in \mathrm{SL}_2(\C)$ with $A\cdot \overline{A}=\left(\begin{smallmatrix} -1 & \phantom{-}0 \\ \phantom{-}0 & -1 \end{smallmatrix}\right)$.
\end{enumerate}
\end{lemma}

\begin{proof}
\ref{H1CCstar} An element of $Z^1((\C^n,+))$ is of the form $\nu\in \C^n$, with $\nu+\overline{\nu}=0$. Choosing $\alpha=\frac{\nu}{2}$, we obtain $\overline{\alpha}=-\alpha$. Whence, $-\alpha+\nu +\overline{\alpha}=0$. This shows $\nu\sim 0$. The same argument applies to $(\C[x_1,\ldots,x_n],+)$.

An element of $Z^1((\C^*,\cdot))$ is of the form $\nu\in \C^{*}$ with $\nu\cdot \overline{\nu}=1$. Hence, $\lvert \nu\rvert=1$. Choosing $\alpha$ with $\alpha^2=\nu$, we obtain $\lvert \alpha\rvert=1$. This implies $\alpha^{-1}\cdot \nu \cdot \overline{\alpha}=1$ and shows that $\nu\sim 1$.

\ref{Mun} As every element $\nu\in\mu_n$ satisfies $\lvert \mu_n\rvert=1$, we have $Z^1(\mu_n)=\mu_n$. Moreover, two elements $\nu,\tau\in \mu_n$ are equivalent if and only if there exists $\alpha\in \mu_n$ such that $\tau=\alpha^{-1}\nu \overline{\alpha}=\nu\alpha^{-2}$, i.e., if and only if $\tau\nu^{-1}$ is a square in $\mu_n$. This implies that $H^1(\mu_n)$ is trivial if $n$ is odd and contains exactly two classes if $n$ is even: the class containing the squares and the one consisting of non-square elements.

\ref{PGL2C} Every element $\tau\in Z^1(\mathrm{PGL}_2(\C))$ is the class of a matrix $A\in \mathrm{SL}_2(\C)$ with 
$A\cdot \overline{A}=\left(\begin{smallmatrix} \epsilon & 0 \\ 0 & \epsilon \end{smallmatrix}\right)$ for some $\epsilon\in \C$. Moreover, $\epsilon=\pm 1$, as $\epsilon^2=\det(A\cdot \overline{A})=1$. We first prove that $\tau$ is equivalent to the class of $\left(\begin{smallmatrix} 0 & \epsilon \\ 1 & 0 \end{smallmatrix}\right)$. For this, choose a $2\times 1$ vector $v$ such that $A\overline{v},v$ are linearly independent. To see that such a vector exists, observe that if $A$ is not diagonal, then we can choose $v=\left(\begin{smallmatrix} 1\\ 0 \end{smallmatrix}\right)$ or $v=\left(\begin{smallmatrix} 0\\ 1 \end{smallmatrix}\right)$. If $A$ is diagonal, then we can choose $v=\left(\begin{smallmatrix} 1\\ \mathbf{i} \end{smallmatrix}\right)$ if $\tau\in\mathrm{PGL}_2(\C)$ is the identity and $v=\left(\begin{smallmatrix} 1\\ 1\end{smallmatrix}\right)$ otherwise. Then, taking the matrix $R=\left(\begin{smallmatrix} v & A\overline{v} \end{smallmatrix}\right)\in \mathrm{GL}_2(\C)$ whose columns are $v$ and $A\overline{v}$ respectively, one checks that $\tau$ is equivalent to the class of
$R^{-1}\cdot A\cdot  \overline{R}=R^{-1}\cdot \left(\begin{smallmatrix} A\overline{v} & A\cdot \overline{A}v \end{smallmatrix}\right)=\left(\begin{smallmatrix} 0 & \epsilon \\ 1 & 0 \end{smallmatrix}\right)\in \mathrm{GL}_2(\C)$.

Now, consider two matrices $A_1,A_2\in \mathrm{SL}_2(\C)$ with $A_i\cdot \overline{A_i}=\left(\begin{smallmatrix} \epsilon_i & 0 \\ 0 & \epsilon_i \end{smallmatrix}\right)$, $\epsilon_i\in \{\pm 1\}$, and suppose that their classes are equivalent $1$-cocycles $\tau_1,\tau_2\in Z^1(\mathrm{PGL}_2(\C))$. To conclude the proof, it remains to show that $\epsilon_1=\epsilon_2$. Since $\tau_1,\tau_2$ are equivalent, there exist $B\in \mathrm{GL}_2(\C)$ and $\mu\in \C^*$ such that $A_2=\mu B^{-1}\cdot  A_1\cdot \overline{B}$. This gives $\left(\begin{smallmatrix} \epsilon_2 & 0 \\ 0 & \epsilon_2 \end{smallmatrix}\right)=A_2\cdot \overline{A_2}=\lvert \mu\rvert^2 B^{-1}\cdot A_1\cdot \overline{A_1}\cdot B=\lvert \mu\rvert^2 \left(\begin{smallmatrix} \epsilon_1 & 0 \\ 0 & \epsilon_1 \end{smallmatrix}\right)$, which implies $\epsilon_1=\epsilon_2$.
\end{proof}

\begin{definition}
If $R$ is a $\C$-algebra, a \textit{real structure} on $R$ is an action of $\mathrm{Gal}(\C/\R)$ on $R$ such that the nontrivial element acts by $\rho\colon \C\mapsto \C, \alpha\mapsto \overline{\alpha}$ on $\C$. This corresponds to giving a ring homomorphism $\rho\colon R\to R$   such that $\rho\circ \rho=\mathrm{id}_R$ and $\rho(\alpha \cdot f)=\overline{\alpha}\cdot \rho(f)$ for each $\alpha\in \C$ and each $f\in R$.   For each such structure, we obtain an action of $\mathrm{Gal}(\C/\R)$ on the group $\Aut_{\C}(R)$ of $\C$-automorphisms  by defining $\overline{f}= \rho \circ f \circ \rho$, for each $f\in \Aut_{\C}(R)$.
\end{definition}

\begin{definition}
If $X$ is a complex algebraic variety, a \textit{real structure} is an action of $\mathrm{Gal}(\C/\R)$ on $X$ such that the action of the nontrivial element is an anti-regular morphism $\rho\colon X\to X$, that is, a morphism of schemes   such that the following diagram commutes:
\[
			\xymatrix@R=0.5cm@C=1.5cm{
			X\ar@{->}[r]^{\rho}\ar@{->}[d]&X\ar@{->}[d]\\
			\mathrm{Spec}(\C) \ar@{->}[r]^{z\mapsto \overline{z}}& \mathrm{Spec}(\C)
			}
			\]	
For each such a real structure, the group $\langle \rho\rangle\simeq\mathrm{Gal}(\C/\R)$ acts on $\Aut_{\C}(X)$ by defining $\overline{f}= \rho \circ f \circ \rho$, for each $f\in \Aut_{\C}(X)$.
\end{definition}
Fixing a real structure $\rho$, we have a bijection between the set of equivalence classes of real structures on $X$ and $H^1(\Aut_{\C}(X))$: Each real structure is of the form $\nu{\circ}\rho$ with $\nu\in Z^1(\Aut_{\C}(X))$ and two real structures $\nu{\circ}\rho$, $\tau{\circ}\rho$ are equivalent if and only if the classes of $\nu$ and $\tau$ in $H^1(\Aut_{\C}(X))$ are equal, which means that $\nu{\circ}\rho$, $\tau{\circ}\rho$ are conjugate with respect to some automorphism $\alpha \in \Aut_{\C}(X)$, i.e., $\tau\circ \rho= \alpha^{-1} \circ (\nu \circ \rho) \circ \alpha$.

\begin{remark}
Giving a real structure on an affine complex variety $X$ is the same as giving a real structure on the $\C$-algebra $\C[X]$ of regular functions. Fixing such a real structure, the group $\mathrm{Gal}(\C/\R)$ acts on $\C[X]$ via ring-automorphisms, and the natural $\C$-anti-isomorphism between $\Aut_{\C}(X)$ and $\Aut_{\C}(\C[X])$ induces an isomorphism of pointed sets \[H^1(\Aut_{\C}(X))\iso H^1(\Aut_{\C}(\C[X])),\] i.e., a bijection sending the identity to the identity.
\end{remark}

\begin{definition}
	A \textit{real form} of a complex algebraic variety $X$ is a real algebraic variety $X_0$ together with a $\C$-isomorphism \[\varphi\colon X_0 \times_{\Spec(\R)} \Spec(\C) \overset{\sim}{\rightarrow} X.\] 
\end{definition}

Real forms and real structures of a complex algebraic variety $X$ correspond to one another: For any real structure $\rho$ on $X$, the variety $X/\langle \rho \rangle$ is a real form of $X$, and, given a real form $(X_0, \varphi)$ of $X$, the map $\varphi\circ(\id \times \Spec(z\mapsto \overline{z}))\circ \varphi^{-1}$ defines a real structure on $X$. We refer to \cite{benzerga2016structures} for a description of the equivalence of categories between quasiprojective complex varieties with a real structure and quasiprojective real varieties.

\begin{example} 
It is an easy exercise to check that $H^1(\Aut(\A^1_\C))$ is trivial, hence that $\A^1_\R$ is the only real form of $\A^1_\C$ up to isomorphism. However, the affine curve $\A^1_\C\setminus \{0\}$ has three different isomorphism classes of real forms, see Proposition~\ref{Prop:Cstar2realforms}. 
\end{example}

\begin{notation}[Usual complex conjugation]\label{notcomplexconj}
For the rest of the text, we shall always denote the standard action of $\mathrm{Gal}(\C/\R)$ on the affine space $\A^n_\C\simeq\C^n$ by $\rho\colon z=(z_1,\ldots,z_n)\mapsto \overline{z}=(\overline{z_1},\ldots,\overline{z_n})$. This provides the standard real structures on $\A^n_\C$ and $\C[\A^n_\C]=\C[x_1,\ldots,x_n]$.

Accordingly, we denote by $\overline{p}=\rho\circ p\circ \rho$ and $\overline{f}=\rho\circ f\circ\rho=(\overline{f_1},\ldots,\overline{f_n})$ the conjugate of a polynomial $p\in\C[\A^n_{\C}]$ and of an endomorphism $f=(f_1,\ldots,f_n)\in\End(\A^n_{\C})$. If $p=\sum_{i_1,\ldots,i_n\ge 0} a_{i_1,\ldots,i_n} x_1^{i_1}\cdots x_n^{i_n}$, then we simply have $\overline{p}=\sum_{i_1,\ldots,i_n\ge 0} \overline{a_{i_1,\ldots,i_n}} x_1^{i_1}\cdots x_n^{i_n}$.
\end{notation}

\section{The surfaces $D_p$}
\label{Sec:DanSurf}

\subsection{Reduced form}

\begin{notation}Given a nonconstant polynomial $p\in\kk[z]$, we denote by $D_p$ the hypersurface in $\A^3_{\kk}=\Spec(\kk[x,y,z])$ defined by the equation $xy=p(z)$.
\end{notation}

\begin{theorem}[{\cite[Lemma 2.10]{Daigle}}]\label{theo:DanielewskiIsoClasses}
Let $\kk$ be a field and let $p,q\in\kk[z]$. 
The surfaces $D_p$ and $D_q$ are isomorphic over $\kk$ if and only if there exist $a,\lambda\in\kk^*$ and $b\in\kk$ such that $p(az+b)=\lambda q(z)$. 
\end{theorem}

\begin{definition}\label{Defi:ReducedForm}
A nonconstant polynomial $p\in\kk[z]$ is called \emph{in reduced form} if $p(z)=z^d+s(z)$ for some integer $d\geq1$ and some polynomial $s\in\kk[z]$ with $\deg(s)\leq d-2$.
\end{definition}

\begin{lemma}\label{Lemma:GoToReduced}If $\kk$ is a field of characteristic zero, then every surface $D_p$ defined over $\kk$ is isomorphic to a surface $D_q$ with $q$ in reduced form.
\end{lemma}

\begin{proof}
For every $p\in\kk[z]$, there exist $\lambda\in \kk^*$ and $\mu\in \kk$ such that the polynomial $q(z)=\lambda p(z+\mu)$ is in reduced form. Then, the affine map $\varphi=(x,\frac{1}{\lambda} y,z+\mu)\in\Aff_3(\kk)$ induces an isomorphism between the hypersurfaces $D_q$ and $D_p$, as $\varphi^*(xy-p(z))=\frac{1}{\lambda}(xy-q(z))$.
\end{proof}

\subsection{Automorphisms}

A list of generators for the automorphism groups of the surfaces $D_p$ has been first given in \cite{ML1}.

\begin{theorem}[{\cite{ML1}}] \label{MLgen}
Let $\kk$ be a field and let $p\in\kk[z]$ be a polynomial of degree at least $2$. Then, every automorphism of the surface $D_p\subset\A^3_{k}$ extends to an automorphism of $\A^3_{k}$. Moreover, the group $\Aut_{\kk}(D_p)$ is generated by the following subgroups:
\begin{itemize}
\item $\left\{(x,y+\frac{p(z+xr(x))-p(z)}{x},z+xr(x))\mid r(x)\in\kk[x]\right\}\simeq(\kk[x],+)$;
\item $\left\{(x,y,z),(y,x,z)\right\}\simeq\Z/2\Z$;
\item $\left\{(ax,b y,cz+d)\mid a,b,c\in \kk^*,d\in\kk, p(cz+d)=ab p(z)\right\}$
\end{itemize}
\end{theorem}

Furthermore, it is a folklore result that $\Aut_{\kk}(D_p)$ is equal to the free product of two subgroups, amalgamated over their intersection, see for instance~\cite{FKZ}. As we did not find the precise statement we need in the literature, we reprove it here. This essentially follows from~\cite[Theorem~5.4.5]{BD} (see also \cite{KL16}, for a slightly weaker statement). 

\begin{theorem} \label{thm: BD11, autos of A1-fibred affine surfaces}
	Let $\kk$ be a field and $p\in\kk[z]$ be a polynomial of degree at least $2$. Let $D_p=\Spec(\kk[x,y,z]/(xy-p(z)))$ and define the subgroups 
\[A_{\kk}(p)=\left\{f\in\Aut_\kk(D_p)\mid \exists g\in \Aff_3(\kk) \colon f=g|_{D_p}\right\}\] 
and
\[B_{\kk}(p)=\left\{\psi_{a,b,c,d,r} \mid a,b,c\in \kk^*, d\in\kk, r\in\kk[x], abp(z)=p(cz+d)\right\},\]
where \[\psi_{a,b,c,d,r}=\left(ax, by+\frac{p(cz+d+xr(x)) -abp(z)}{ax}, cz+d+xr(x)\right).\]
 Then, the automorphism group of $D_p$ is the free product 
	\[\Aut_{\kk}(D_p)=A_{\kk}(p) \Asterisk_{\cap} B_{\kk}(p)\]
	of $A_{\kk}(p)$ and $B_{\kk}(p)$ amalgamated over their intersection $A_{\kk}(p)\cap B_{\kk}(p)$, which is the set of all elements $\psi_{a,b,c,d,r}$, with $a,b,c\in \kk^*, d\in\kk$ such that $abp(z)=p(cz+d)$, and $r\in\kk$ such that $r=0$ if $\deg(p)\ge 3$.
\end{theorem}

\begin{remark}\label{Rema:Fibration} One can check, using the birational morphism $D_p\to \A^2_{\kk}$, $(x,y,z)\mapsto (x,z)$, that $B_{\kk}(p)$ consists of all automorphisms of $D_p$ that preserve the fibration $(x,y,z)\mapsto x$.
\end{remark}

\begin{proof}
First, we check that the set $B_{\kk}(p)$ is indeed a subgroup of $\Aut_\kk(D_p)$. For this, it suffices to remark that $\psi_{a,b,c,d,r}$ defines an endomorphism of $\A^3_{\kk}$ satisfying $\psi_{a,b,c,d,r}^*(xy-p(z))=ab(xy-p(z))$, and to compute $\psi_{1,1,1,0,0}=\id_{\A^3_{\kk}}$ and \[\psi_{a,b,c,d,r}\circ \psi_{a',b',c',d',r'}=\psi_{aa',bb',cc',cd'+d,cr'(x)+a'r(a'x)}\] 
for all $\psi_{a,b,c,d,r},\psi_{a',b',c',d',r'}\in B_{\kk}(p)$.

Let us consider the open embedding $\A^3_{\kk}\hookrightarrow \mathbb{P}^3_{\kk}$, $(x,y,z)\mapsto [1:x:y:z]$ and denote by $X_p$ the closure of $D_p$ in $\mathbb{P}^3$. Writing $s=\deg(p)$ and $p(z)=\sum_{i=0}^s p_i z^i$ with $p_0,\ldots,p_s\in\kk$ and $p_s \neq 0$, we obtain that $X_p$ is the hypersurface in $\mathbb{P}^3$ given by the equation $w^{s-2}xy=\sum_{i=0}^s p_i w^{s-i}z^i$. So, $C_p=X_p\setminus D_p$ is either the conic defined by $\{w=0, xy=p_2 z^2\}$ in the case where $s=2$, or the line given by $\{w=z=0\}$ in the case where $s\ge 3$. In both cases, $C_p$ is a curve isomorphic to $\mathbb{P}^1_{\kk}$, that contains the point $q=[0:0:1:0]$.

We will prove the two following statements. 
\begin{enumerate}[leftmargin=*]
\item \label{Claim1}  The birational map $\hat{\beta}$ of $X_p$ induced by any $\beta\in B_{\kk}(p)\setminus A_{\kk}(p)$ contracts $C_p\setminus \{q\}$ onto $q$. 
\item  \label{Claim2} The birational map $\hat{\alpha}$ of $X_p$ induced by any $\alpha\in A_{\kk}(p)$ preserves the curve $C_p$, and if it fixes the point $q$, then $\alpha\in A_{\kk}(p)\cap B_{\kk}(p)$.
\end{enumerate}

Before proving them, let us show that  $\Aut_{\kk}(D_p)=A_{\kk}(p) \Asterisk_{\cap} B_{\kk}(p)$ follows from these two claims. Recall 
 that by Theorem~\ref{MLgen}, $\Aut_{\kk}(D_p)$ is generated by $A_{\kk}(p)$ and $B_{\kk}(p)$. Letting $m\ge 1$, $\alpha_1,\ldots,\alpha_{m-1} \in A_{\kk}(p)\setminus B_{\kk}(p)$ and $\beta_1,\ldots,\beta_m\in B_{\kk}(p)\setminus A_{\kk}(p)$, it then suffices to prove that \[\varphi= \beta_m \circ\alpha_{m-1}\circ\cdots\circ \alpha_1\circ\beta_1\not\in A_{\kk}(p).\] 
For this, we prove by induction on $m$ that the extension of $\varphi\in \Aut_{\kk}(D_p)$ to a birational map $\hat\varphi\in \Bir_{\kk}(X_p)$ contracts $C_p\setminus \{q\}$ onto $q$.
 For $m=1$, this is given by \ref{Claim1}. For $m\geq2$, write $\varphi=\beta_m\circ\alpha_{m-1}\circ\varphi'$. The result follows, since the extensions $\hat\beta_m,\hat\alpha_{m-1},\hat\varphi'$ are elements of $\Bir_{\kk}(X_p)$ such that $\hat\varphi'(C_p\setminus \{q\})=\{q\}$ (by induction hypothesis), $\hat\alpha_{m-1}(q)\in C_p\setminus \{q\}$ (by \ref{Claim2}) and $\hat\beta_m(C_p\setminus \{q\})=\{q\}$ (by \ref{Claim1}). 

We now prove \ref{Claim1}. First, remark that an automorphism $\psi_{a,b,c,d,r}$ is in $A_{\kk}(p)$ if $r=0$, or if $\deg(r)=0$ and $\deg(p)=2$. Therefore, we consider an automorphism $\psi=\psi_{a,b,c,d,r}$ with either $\deg(r)\geq1$, or with $\deg(r)=0$ and $\deg(p)\geq3$.

In both cases, the second component of the birational map $\hat{\psi}\in\Bir_{\kk}(X_p)$ induced by $\psi$ is of degree $D\coloneqq\deg(p)\cdot (\deg(r)+1)-1>\deg(r)+1$, strictly greater than the degree of any other component of the map, and its leading term is $\xi x^D$ for some $\xi\in \kk^*$.

 Extending $\psi$ to a rational map $\tilde{\psi}\colon \mathbb{P}^3\dasharrow \mathbb{P}^3$ by homogenizing its components, we obtain $\tilde{\psi}([0:x:y:z])=[0:0:\xi x^D:0]$ for any $[0:x:y:z] \in C_p$. As every point of $C_p\setminus \{q\}$ satisfies $x\neq0$, the equality $\hat{\psi}(C_p\setminus \{q\})=\{q\}$ follows. This proves~\ref{Claim1}.

We remark that we have proven above that every map $\hat\psi_{a,b,c,d,r}\in\Bir_{\kk}(X_p)$ is not an automorphism if $\deg(r)\geq1$ or if $\deg(r)=0$ and $\deg(p)\geq3$. In particular, it is not an element of $A_{\kk}(p)$ in these cases. Hence, we get the desired description of $A_{\kk}(p)\cap B_{\kk}(p)$.

Finally, it remains to prove \ref{Claim2}. Let $\alpha\in A_{\kk}(p)$. As it is the restriction of an element of $\Aff_3(k)$, which itself is the restriction of an element $\tilde\alpha\in \Aut(\mathbb{P}^3_{\kk})$ that preserves the curve $C_p$, the automorphism $\alpha$ induces a map $\hat{\alpha}\in\Bir_{\kk}(X_p)$ that preserves $C_p$. Suppose that $\hat{\alpha}(q)=q$. Then, the birational morphism $\kappa\colon D_p\to \A^2_{\kk}, (x,y,z)\mapsto (x,z)$ conjugates $\alpha$ to an affine automorphism $\alpha'\in \Aut(\A^2_{\kk})$, because this morphism is the restriction of the projection $\mathbb{P}^3_{\kk}\dasharrow \mathbb{P}^2_{\kk}$, $[w:x:y:z]\mapsto [w:x:z]$ from the point $q$. For each $(x,z)\in \A^2_{\kk}$, the fibre $\kappa^{-1}(x,z)$ consists of one single point if and only if $x\neq 0$. Hence, $\alpha'$ is of the form $(x,z)\mapsto (ax,cz+d+r_0x)$ for some $a,c\in \kk^*$ and some $d,r_0\in\kk$. This gives $\alpha=(ax,by+h(x,z),cz+d+r_0 x)$ for some $b\in \kk^*$ and some $h\in\kk[x,z]$ of degree $1$. As $\alpha^*(xy-p(z))=abxy+axh(x,z)-p(cz+d+r_0 x)$ lies in the ideal generated by $xy-p(z)$, it must be equal to $ab(xy-p(z))$. This implies, by setting $x=0$, that $abp(z)=p(cz+d)$, and then that $h(x,z)=\frac{p(cz+d+r_0x) -abp(z)}{ax}$; hence $\alpha\in B_{\kk}(p)$.
\end{proof}

The aim of the next three results is to give a precise description of the subgroup $A_{\kk}(p)$ of ``affine'' automorphisms of a surface $D_p$. We start by the case where $\deg(p)\geq3$.

\begin{lemma}\label{AcapBdeg3}Let $\kk$ be a field and $p\in\kk[z]$ with $\deg(p)\ge 3$.
Then,
\[A_{\kk}(p) =(A_{\kk}(p)\cap B_{\kk}(p)) \rtimes \langle (y,x,z) \rangle,\]
where 
\[A_{\kk}(p)\cap B_{\kk}(p)=\left\{(ax,by,cz+d) \mid a,b,c\in \kk^*, d\in\kk, abp(z)=p(cz+d) \right\}.\]
\end{lemma}

\begin{proof}By Theorem~\ref{thm: BD11, autos of A1-fibred affine surfaces}, we have
\[A_{\kk}(p)\cap B_{\kk}(p)=\left\{(ax,by,cz+d) \mid a,b,c\in \kk^*, d\in\kk, abp(z)=p(cz+d) \right\}.\]

As the involution $(y,x,z)$ is an element of $A_{\kk}(p)\setminus B_{\kk}(p)$ that normalises $A_{\kk}(p)\cap B_{\kk}(p)$, the subgroup of $A_{\kk}(p)$ generated by $A_{\kk}(p)\cap B_{\kk}(p)$ and $(y,x,z)$ is isomorphic to $(A_{\kk}(p)\cap B_{\kk}(p)) \rtimes \langle (y,x,z) \rangle$. It remains to see that every element $\alpha\in A_{\kk}(p)$ is in that subgroup, i.e., is of the form $\alpha=(ax,by,cz+d)$ or $\alpha=(ay,bx,cz+d)$ for some $a,b,c\in\kk^*$ and $d\in\kk$. 

Write $\alpha=(\ell_1,\ell_2,\ell_3)$, where $\ell_1,\ell_2,\ell_3\in\kk[x,y,z]$ are of degree $1$. Then, 
\[\ell_1\ell_2-p(\ell_3)=\alpha^*(xy-p(z))=\mu(xy-p(z))\]
for some $\mu\in \kk^*$. Since $\deg(p)\geq3$ and $\deg(\ell_1\ell_2)=2$, we obtain that $\ell_3=cz+d$ for some $c\in \kk^*, d\in\kk$ and we have that 
\[\ell_1\ell_2=\mu xy+p(cz+d)-\mu p(z).\] 
Observe that the right-hand side of the above equality is an irreducible polynomial, unless $p(cz+d)-\mu p(z)=0$. Thus, $p(cz+d)=\mu p(z)$ and $\ell_1\ell_2=\mu xy$. In turn, the latter equality
implies that either $\ell_1=ax$ and $\ell_2=by$ or $\ell_1=ay$ and $\ell_2=bx$ for some $a,b\in\kk^*$ with $ab=\mu$.
\end{proof}

We now investigate the case where $\deg(p)=2$.

\begin{lemma}\label{PGL2Z2}
Let $\kk$ be a field of characteristic not $2$, and let $p=z^2-1=(z-1)(z+1)\in\kk[z]$. 
The surface $D_p=\Spec(\kk[x,y,z]/(xy-p(z)))$ is isomorphic to $(\mathbb{P}^1_{\kk}\times \mathbb{P}^1_{\kk})\setminus \Delta$, where $\Delta$ denotes the diagonal, via
\[\begin{array}{ccc}
(\mathbb{P}^1_{\kk}\times \mathbb{P}^1_{\kk})\setminus \Delta&\iso & D_p \\
([a:b],[c:d])&\mapsto &\left(\frac{2ac}{ad-bc},\frac{2bd}{ad-bc},\frac{ad+bc}{ad-bc}\right).
\end{array}\]
Moreover, $A_{\kk}(p)$ is isomorphic to $ \mathrm{PGL}_2(\kk)\times \langle \sigma \rangle$, where $\sigma=(-x,-y,-z)\in \Aut_{\kk}(D_p)$ acts on $\mathbb{P}^1_{\kk}\times \mathbb{P}^1_{\kk}$ via the exchange of the two factors and where $\mathrm{PGL}_2(\kk)$ acts diagonally on $(\mathbb{P}^1_{\kk}\times \mathbb{P}^1_{\kk})\setminus \Delta$ and via
\[\begin{array}{ccc}
\mathrm{PGL}_2(\kk) \times D_p & \to & D_p\\
 \left(\begin{smallmatrix} \alpha & \beta \\ \gamma & \delta\end{smallmatrix}\right), \left(\begin{smallmatrix} x \\ y\\ z\end{smallmatrix}\right) &\mapsto &  \frac{1}{\alpha\delta-\beta\gamma}\left(\begin{smallmatrix} \alpha^2& \beta^2& 2\alpha\beta \\ \gamma^2& \delta^2& 2\gamma\delta \\ \alpha\gamma &\beta\delta& \alpha\delta+\beta\gamma \end{smallmatrix}\right) \cdot  \left(\begin{smallmatrix} x \\ y\\ z\end{smallmatrix}\right) \end{array}\]
 on $D_p$.
\end{lemma}
\begin{proof}
As $\mathrm{char}(\kk)\neq2$, we may consider $\A^3_{\kk}$ embedded into $\mathbb{P}^3_{\kk}$, via the open embedding $(x,y,z)\mapsto [x:y:z+1:z-1]$, and obtain that $D_p=Q\setminus H$, where $Q,H\subset \mathbb{P}^3_{\kk}$ are given respectively by $x_0x_1=x_2x_3$ and $x_2=x_3$.

We then use the classical isomorphism $\mathbb{P}^1_{\kk}\times  \mathbb{P}^1_{\kk}\iso Q$, $([a:b],[c:d])\mapsto [ac:bd:ad:bc]$, which restricts to the isomorphism $(\mathbb{P}^1_{\kk}\times \mathbb{P}^1_{\kk})\setminus \Delta\iso Q\setminus H=D_p$ described in the statement.

By definition, $A_{\kk}(p)=\left\{f\in\Aut_\kk(D_p)\mid \exists g\in \Aff_3(\kk) \colon f=g|_{D_p}\right\}$ corresponds to the group of automorphisms of $\mathbb{P}^3$ which preserve $H$ and $Q$, and thus to the group of automorphisms of $Q$ that preserve $Q\cap H$; it is conjugate  via the above isomorphism to the group of automorphisms of $\mathbb{P}^1_{\kk}\times \mathbb{P}^1_{\kk}$ that preserve the diagonal.

As $\Aut(\mathbb{P}^1_{\kk}\times \mathbb{P}^1_{\kk})=(\mathrm{PGL}_2(\kk)\times \mathrm{PGL}_2(\kk))\rtimes \langle \sigma \rangle$, where $\sigma$ is the exchange of the two factors, that corresponds to $(-x,-y,-z)\in \Aut_{\kk}(D_p)$, we obtain that $A_{\kk}(p)$ corresponds, via the isomorphism, to the group $\mathrm{PGL}_2(\kk)\times \langle \sigma \rangle$, where $\mathrm{PGL}_2(\kk)$ acts diagonally on $\mathbb{P}^1_{\kk}\times \mathbb{P}^1_{\kk}$. Conjugating the action gives the explicit description of the action of $\mathrm{PGL}_2(\kk)$ on $D_p$.
\end{proof}
\begin{lemma}\label{PGL2k}
Let $\kk$ be a field of characteristic  not $2$, and let $p=z^2\in \kk[z]$. The group $A_{\kk}(p)$ is isomorphic to $ \mathrm{PGL}_2(\kk)\times \kk^*$, and the action of this latter group on the  surface $D_p=\Spec(\kk[x,y,z]/(xy-z^2))$ is 
\[\begin{array}{ccc}
(\mathrm{PGL}_2(\kk)\times \kk^*) \times D_p & \to & D_p\\
\left( \left(\begin{smallmatrix} \alpha & \beta \\ \gamma & \delta\end{smallmatrix}\right), \mu \right), \left(\begin{smallmatrix} x \\ y\\ z\end{smallmatrix}\right) &\mapsto &  \frac{\mu}{\alpha\delta-\beta\gamma}\left(\begin{smallmatrix} \alpha^2& \beta^2& {2}\alpha\beta \\ \gamma^2& \delta^2& {2}\gamma\delta \\ \alpha\gamma &\beta\delta& \alpha\delta+\beta\gamma \end{smallmatrix}\right) \cdot  \left(\begin{smallmatrix} x \\ y\\ z\end{smallmatrix}\right). \end{array}\]
\end{lemma}
\begin{proof}
As observed in Lemma~\ref{PGL2Z2}, the above formula gives an embedding $\mathrm{PGL}_2(\kk)\hookrightarrow \mathrm{GL}_3(\kk)$ whose action on $\A^3_{\kk}$ preserves $xy-z^2-1$, and thus also $xy-z^2$. Its image moreover lies in $\mathrm{SL}_3(\kk)$. The action of $\kk^*$ on $\A^3_{\kk}$ by homotheties gives another embedding $\kk^*\hookrightarrow \mathrm{GL}_3(\kk)$. Since both groups commute and have a trivial intersection, we get an embedding $\varphi\colon\mathrm{PGL}_2(\kk)\times \kk^*\hookrightarrow A_{\kk}(p)$.

It remains to see that every element $f\in A_{\kk}(p)$ lies in the image of $\varphi$.  As it is the only singular point of $D_p$, the point $(0,0,0)\in D_p\subset \A^3_{\kk}$   is fixed by any $f\in A_{\kk}(p)$. Hence, $f=g|_{D_p}$ for some $g\in \mathrm{GL}_3(\kk)$ whose action on $\mathbb{P}^2$ preserves the conic $\Gamma$ given by $xy=z^2$, isomorphic to $\mathbb{P}^1$ via $[u:v]\mapsto [u^2:v^2:uv]$.  The induced action of $g$ on $\mathbb{P}^1$ is of the form $[u:v]\mapsto [\alpha u+\beta v:\gamma u+\delta v]$ for some $R=\left(\begin{smallmatrix} \alpha & \beta \\ \gamma & \delta\end{smallmatrix}\right)$ in $\mathrm{PGL}_2(\kk)$. Hence, the action of $g$  on $\mathbb{P}^1$ coincides with that of the image of $R$ in $\mathrm{PGL}_2(\kk)\subset \mathrm{SL}_3(\kk)$, i.e.,~with that of $\varphi((R,1))$. Hence, the  map $f\circ \varphi((R,1))^{-1}\in A_{\kk}(p)$ acts trivially on $\Gamma$ and thus on $\mathbb{P}^2$ (a nontrivial automorphism of $\mathbb{P}^2$ only fixes points and lines), and is then a homothety.
\end{proof}
\subsection{Existence of real forms}
\begin{proposition}\label{Prop:ExistencerealstructureDanielewski}
Let $p\in \C[z]\setminus \C$ be a nonconstant polynomial. The following conditions are equivalent:
\begin{enumerate}[leftmargin=*]
\item\label{ExRD1}
The complex affine surface $D_p=\Spec(\C[x,y,z]/(xy-p(z)))$ admits a real structure.
\item\label{ExRD2}
There exist $a,\lambda \in \C^*$, $b\in \C$, such that $\lambda p(az+b)\in \R[z]$.
\item\label{ExRD3}
There exists $q\in \R[z]$ such that the complex affine surfaces $D_p=\Spec(\C[x,y,z]/(xy-p(z)))$ and $D_q=\Spec(\C[x,y,z]/(xy-q(z)))$ are isomorphic.
\end{enumerate}
\end{proposition}

\begin{proof}
The equivalence between \ref{ExRD2} and \ref{ExRD3} follows from Theorem~\ref{theo:DanielewskiIsoClasses}.

The implication $\ref{ExRD3}\Rightarrow \ref{ExRD1}$ follows from the fact that $(x,y,z)\mapsto (\overline{x},\overline{y},\overline{z})$ is a real structure on $D_q=\Spec(\C[x,y,z]/(xy-q(z)))$, since $q\in \R[z]$.

It remains to prove $\ref{ExRD1}\Rightarrow \ref{ExRD2}$. Applying a suitable affine automorphism of the form $(\lambda x,y,az+b)$ we can assume that $p$ is in reduced form. Let $d=\deg(p)\geq1$. Since $\ref{ExRD2}$ is satisfied when $p=z^d$, we may further assume that $p$ is not a monomial.

We take a real structure on $D_p$ which we can write as \[(x,y,z)\mapsto (\overline{f_1(x,y,z)},\overline{f_2(x,y,z)},\overline{f_3(x,y,z)}),\] for some polynomials $f_1,f_2,f_3\in \C[x,y,z]$. This provides an isomorphism of complex affine surfaces
\[\begin{array}{rcl}
D_p & \mapsto & D_{\overline{p}}\\
(x,y,z)&\mapsto & (f_1(x,y,z),f_2(x,y,z),f_3(x,y,z)).\end{array}\]
Hence, by Theorem~\ref{theo:DanielewskiIsoClasses}, there exist $a,\lambda \in \C^*$ and $b\in \C$ such that $\overline{p}(z)=\lambda p(az+b)$. Since $p$ is in reduced form and is not a monomial, we have $b=0$, $\lambda=a^{-d}$ and $\lvert a\rvert=1$. Let $\alpha\in\C^*$ be such that $\alpha^2=a$. We now conclude the proof by showing that the polynomial $q(z)=\alpha^{-d}p(\alpha z)$ lies in $\R[z]$.

Indeed, since $\rvert a \lvert=\rvert \alpha \lvert=1$, we get
\[\overline{q}(z)=\overline{\alpha^{-d}}\cdot\overline{p}(\overline{\alpha} z)=\alpha^d\lambda p(a\overline{\alpha} z)=\alpha^d a^{-d} p(a\alpha^{-1} z)=\alpha^{-d}p(\alpha z)=q(z),\]
as desired.
\end{proof}

By Proposition~\ref{Prop:ExistencerealstructureDanielewski}, we may assume a  surface $D_p$ to have no real forms or its defining polynomial $p$ to lie in $\mathbb{R}[z]$. We shall classify the number of real forms for the latter case. We first prove that if $p \in \R[z]$, then both subgroups $A_\C(p)$ and $B_\C(p)$ of $\Aut_{\C}(D_p)$ defined in Theorem~\ref{thm: BD11, autos of A1-fibred affine surfaces} are invariant under the action of $\rho\colon (x,y,z) \mapsto (\overline{x}, \overline{y}, \overline{z})$.

\begin{lemma} \label{lemma: invariant subgroups of autom-gp}
	Assume that $p\in \R[z]$. Then, the subgroups $A_\C(p)$ and $B_\C(p)$ of $\Aut_{\C}(D_p)$ given in Theorem~$\ref{thm: BD11, autos of A1-fibred affine surfaces}$ are invariant under the action of $\Gal(\C/\R)$.
\end{lemma}
\begin{proof}
	Since $p$ is real, we have $\rho(D_p)=D_p$, and thus $\overline{f}(D_p)=\rho\circ f \circ\rho(D_p)=D_p$ for all $f\in\Aut(\A^3_\C)$ satisfying $f(D_p)=D_p$. Since any element of $A_\C(p)$ comes from the restriction of an element of $\Aff_3(\C)$, this implies that $A_\C(p)$ is invariant under the action of $\Gal(\C/\R)$. Similarly, as $\overline{\psi_{a,b,c,d,r}}=\psi_{\bar{a},\bar{b},\bar{c},\bar{d},\bar{r}}$ for all $\psi_{a,b,c,d,r}\in B_\C(p)$, the group $B_\C(p)$ is also invariant under the action of $\Gal(\C/\R)$.
\end{proof}

The next result shows that it will be actually sufficient to compute the cohomology set $H^1(A_\C(p))$ to determine all real forms of $D_p$.

\begin{lemma}\label{Lemm:H1B-new}
Let $p \in \C[z]$ be a polynomial with $\deg(p)\geq 2$. The homomorphisms of pointed sets
\[\begin{array}{ccc}
H^1(A_\C(p)\cap B_\C(p))&\to &H^1(B_\C(p)),\\
 H^1(A_\C(p))&\to& H^1(\Aut_\C(D_p))\end{array}\]
given by the inclusions $A_\C(p)\cap B_\C(p)\hookrightarrow B_\C(p)$ and $A_\C(p) \hookrightarrow\Aut_\C(D_p)$ are isomorphisms of pointed sets.
\end{lemma}
\begin{proof}
Recall that by definition every element $\psi$ of $B_\C(p)$ is of the form $\psi=\psi_{a,b,c,d,r}$ for some $a,b,c\in\C^*$, $d\in\C$ and $r\in\C[x]$ such that $p(cz+d)=abp(z)$. Moreover, for all $\psi_{a,b,c,d,r},\psi_{a',b',c',d',r'}$ in $B_\C(p)$, we have 
\[\psi_{a,b,c,d,r}\circ\psi_{a',b',c',d',r'}=\psi_{aa',bb',cc',cd'+d,cr'(x)+a'r(a'x)}\]
and 
\[\psi_{a,b,c,d,r}=\psi_{a',b',c',d',r'} \text{ if and only if } a=a',b=b',c=c',d=d',r=r'.\]
The latter claim can be proven using the birational morphism $D_p\to \A^2$, $(x,y,z)\mapsto (x,z)$, or by saying that if the two maps are equal, then $\psi_{a,b,c,d,r}$ and $\psi_{a',b',c',d',r'}$ have the same components modulo $xy-p(z)$.
Remark also that any element of $B_\C(p)$ of the form $\psi_{a,b,c,d,0}$ belongs to $A_\C(p)$.

By \cite[Theorem 1]{Kambayashi}, the fact that $\Aut_\C(D_p)$ is the free product of $A_\C(p)$ and $B_\C(p)$ amalgamated over their intersection as in Theorem~\ref{thm: BD11, autos of A1-fibred affine surfaces} implies that we have the following cocartesian diagram of morphisms of pointed sets.
\[
			\xymatrix@R=0.5cm@C=1.5cm{
			H^1(A_\C(p)\cap B_\C(p))\ar@{->}[r]\ar@{->}[d]&H^1(A_\C(p))\ar@{->}[d]\\
			H^1(B_\C(p)) \ar@{->}[r]&H^1(\Aut_\C(D_p))
			}
			\]	 
Therefore, it suffices to prove that $H^1(A_\C(p)\cap B_\C(p))\to H^1(B_\C(p))$ is a bijection to obtain that $H^1(A_\C(p))\to H^1(\Aut_\C(D_p))$ is a bijection.

For this, we will show that: 
\begin{enumerate}[leftmargin=*]
\item\label{BAcapB1} Each element of $Z^1(B_\C(p))$ is equivalent to an element  $\psi\in Z^1(A_\C(p)\cap B_\C(p))$ of the form $\psi=\psi_{1,b,c,d,0}$. 
\item\label{BAcapB2} Two such elements $\psi_{1,b,c,d,0}$, $\psi_{1,b',c',d',0}$ of $Z^1(A_\C(p)\cap B_\C(p))$ are equivalent in $B_\C(p)$ if and only if they are equivalent in $A_\C(p)\cap B_\C(p)$.
\end{enumerate}

Let $\tau=\psi_{a,b,c,d,r}$ be a $1$-cocycle in $Z^1(B_\C(p))$. This implies $a\overline{a}=1$, as $\tau\circ\overline{\tau}=\id_{D_p}$. Therefore, we can find $\varepsilon\in\C^*$ with $\varepsilon^{2}=a$ and define $\theta=\psi_{\varepsilon,\varepsilon^{-1},1,0,0}=(\varepsilon x,\varepsilon^{-1}y,z)\in A_\C(p)\cap B_\C(p)$. Then, 
\begin{align*}
\tilde{\tau}&=\theta^{-1}\circ\tau\circ\overline{\theta}\\
&=\psi_{\varepsilon^{-1},\varepsilon,1,0,0}\circ\psi_{a,b,c,d,r}\circ\psi_{\varepsilon^{-1},\varepsilon,1,0,0}\\
&=\psi_{\varepsilon^{-2}a,\varepsilon^2b,c,d,\varepsilon^{-1} r(\varepsilon^{-1}x)}\\
&=\psi_{1,ab,c,d,\varepsilon^{-1} r(\varepsilon^{-1}x)}
\end{align*}
is a $1$-cocycle in $Z^1(B_\C(p))$ equivalent to $\tau$.

Denote $s(x)=r(\varepsilon^{-1}x)\in\C[x]$. Computing the third component of $\tilde{\tau}\circ\overline{\tilde{\tau}}=\id_{D_p}$, we see that $c\overline{s}(x)+s(x)=0$. Define $\psi=\psi_{1,1,1,0,\frac{1}{2}s}$. Then, $\tau'=\psi^{-1}\circ\tilde{\tau}\circ\overline{\psi}$ is a $1$-cocycle in $Z^1(B_\C(p))$ equivalent to $\tau$. Moreover, one checks that 
\begin{align*}
\tau'&=\psi^{-1}\circ\tilde{\tau}\circ\overline{\psi}\\
&=\psi_{1,1,1,0,-\frac{1}{2}s(x)}\circ\psi_{1,ab,c,d,s(x)}\circ\psi_{1,1,1,0,\frac{1}{2}\overline{s}(x)}\\
&=\psi_{1,ab,c,d,\frac{1}{2}s(x)}\circ\psi_{1,1,1,0,\frac{1}{2}\overline{s}(x)}\\
&=\psi_{1,ab,c,d,c\frac{1}{2}\overline{s}(x)+\frac{1}{2}s(x)}\\
&=\psi_{1,ab,c,d,0}.
\end{align*}
This proves \ref{BAcapB1}.

Now, let $\tau=\psi_{1,b,c,d,0}$ and $\sigma=\psi_{1,b',c',d',0}$ be two elements in $Z^1(B_\C(p))$ and suppose that $\varphi^{-1}\circ\tau\circ\overline{\varphi}=\sigma$ for some $\varphi=\psi_{\alpha,\beta,\gamma,\delta,\pi}$ in $B_\C(p)$. It is then straightforward to check that $\psi^{-1}\circ\tau\circ\overline{\psi}=\sigma$, where $\psi$ is the element of $A_\C(p)\cap B_\C(p)$ defined by $\psi=\psi_{\alpha,\beta,\gamma,\delta,0}$. This proves \ref{BAcapB2}.
\end{proof}

\subsection{Cohomology set of the group $A_\C(p)$}

 We first deal with the case where $\deg(p)=2$. In view of Lemma~\ref{PGL2Z2} and Lemma~\ref{PGL2k}, we proceed in two distinct cases.

\begin{lemma}\label{Az21}
If $p=z^2-1$, then $H^1(A_\C(p))$ contains exactly four elements, namely the classes of
$(x,y,z)$, $(-x,-y,-z)$, $(y,x,-z)$, $(-y,-x,z)$.
\end{lemma}
\begin{proof}
Lemma~\ref{PGL2Z2} provides an explicit isomorphism $ \mathrm{PGL}_2(\C)\times \langle \sigma \rangle\iso A_\C(p)$, where $\sigma$ is an involution, the action of $\Gal(\C/\R)$ on $\langle \sigma \rangle\simeq\Z/2$ is trivial and the one on $\mathrm{PGL}_2(\C)$ is the standard one. As by Lemma~\ref{Lem:H1CCstar}\ref{PGL2C}, $H^1(\mathrm{PGL}_2(\C))$ consists of two elements, which are the class of the identity and that of $M=\left(\begin{smallmatrix} 0 & -1 \\ 1 & \phantom{-}0 \end{smallmatrix}\right)$, we find that $H^1(A_\C(p))$ consists of exactly four elements, which are the classes of the images  of $(\mathrm{id},\mathrm{id}), (\mathrm{id},\sigma), (M,\mathrm{id})$ and $(M,\sigma)$ under the above isomorphism. It moreover follows from the explicit action of $ \mathrm{PGL}_2(\C)\times \langle \sigma \rangle$ on $D_p$ given in Lemma~\ref{PGL2Z2} that these four images are equal to
$(x,y,z)$, $(-x,-y,-z)$, $(y,x,-z)$ and $(-y,-x,z)$, respectively.
\end{proof}

\begin{lemma}\label{Az2}
If $p=z^2$, then $H^1(A_\C(p))$ contains exactly two elements, which are the classes of
$(x,y,z)$ and of $(y,x,-z)$. 
\end{lemma}
\begin{proof}
Lemma~\ref{PGL2k} provides an explicit isomorphism $ \mathrm{PGL}_2(\C)\times \C^*\iso A_\C(p)$. As $H^1(\mathrm{PGL}_2(\C))$ consists of two elements, which are the class of the identity and that of $M=\left(\begin{smallmatrix} 0 & -1 \\ 1 & \phantom{-}0 \end{smallmatrix}\right)$ --- see Lemma~\ref{Lem:H1CCstar}\ref{PGL2C}) --- and as $H^1(\C^*)=\{1\}$, the pointed set $H^1(A_\C(p))$ contains exactly two elements, which are the classes of the identity and that of the image of $(M,1)$ under the above isomorphism. This latter is equal to the class of $(y,x,-z)$, compare with Lemma~\ref{PGL2k}.
\end{proof}

To describe $H^1(A_\C(p))$ when $\deg(p)\ge 3$, we will need the group $H_p\in \Aut(\A^1_\C)$ associated to $p$. It corresponds to the group of symmetries of the polynomial. 

\begin{definition}\label{Defi:Hp}
Let $p\in \C[z]$ be a polynomial. We denote by $H_p\subseteq \Aut(\A^1_\C)=\Aut_{\C}(\Spec(\C[z]))$ the subgroup \[H_p=\{(cz+d) \mid c\in \C^*, d\in \C, \exists \lambda\in\C^*: p(cz+d)=\lambda p(z)\}.\]
\end{definition}

\begin{lemma}\label{lemma:Hp}\label{lemma: H^1(H_p)}
Let $p\in\C[z]$ be in reduced form. 
\begin{enumerate}[leftmargin=*]
\item\label{H1Hp1}
If $p$ has a unique root, then $p=z^d$ is a monomial and $H_p=\{(cz)\mid c\in \C^*\}$. In particular, $H_p$ is then isomorphic to $\C^*$ and $H^1(H_p)$ contains only one element, namely the class of $(z)$.
\item\label{H1Hp2}
If $p$ has at least two roots, then $H_p=\{(cz)\mid c\in \C^*, c^n=1\}$ is cyclic of finite order $n\geq1$. In particular, $H^1(H_p)$ contains either a single element when $n$ is odd or two elements when $n$ is even, namely the classes of $(z)$ and $(cz)$ where $c$ denotes any primitive $n$-th root of unity. Moreover, $p$ is of the form $p(z)=z^mq(z^n)$ for some integer $m\geq0$ and some polynomial $q\in\C[t]$ with $q(0)\neq0$. 
\end{enumerate}
\end{lemma} 
\begin{proof}
\ref{H1Hp1} Recall that $H^1(\C^*)$ is trivial by Lemma~\ref{Lem:H1CCstar}.

\ref{H1Hp2} Let $p(z)=\sum_{i=0}^{\ell}p_iz^i\in\C[z]$ with $p_{\ell}=1$ and $p_{\ell-1}=0$ and suppose that 
$p$ is not a monomial. Suppose that $c,\lambda\in\C^*$ and $d\in\C$ are such that $p(cz+d)=\lambda p(z)$. Then, $d=0$ because $p_{\ell-1}=0$. Moreover, for any $i$, $j$ with $p_i,p_j\neq0$, we find $c^i=\lambda=c^j$. This implies that $c$ is of finite order, say $n\geq1$, and that $i\equiv j\pmod{n}$. Hence, $p$ is of the form $p(z)=z^mq(z^n)$ as claimed in the statement. In turn, $H_p=\{(cz)\mid c\in\C^*, c^n=1\}$ is cyclic of order $n$. Finally, the claims about $H^1(H_p)$ follow from Lemma~\ref{Lem:H1CCstar}.
\end{proof}

\begin{lemma}\label{Lemm:H1A}
	Let $p\in \R[z]$ be a polynomial of degree at least $3$ in reduced form. Then, the following holds:
\begin{enumerate}[leftmargin=*]
\item\label{H1A2}
If $H_p$ is infinite and $\deg(p)$ is odd, then $H^1(A_\C(p))$ contains exactly two elements, namely the classes of
\[(x,y,z),(y,x,z).\]
\item\label{H1A3}
If $H_p$ is infinite and $\deg(p)$ is even, or $H_p$ is finite of odd order, then $H^1(A_\C(p))$ contains exactly three elements, namely the classes of
\[(x,y,z),(y,x,z), (-y,-x,z)\]
\item\label{H1A4}
If $H_p$ is of even order $n\ge 2$ and $\deg(p)$ is odd, then $H^1(A_\C(p))$ contains exactly four elements, namely the classes of
\[(x,y,z),(ax,ay,cz),(y,x,z),( ay,a x,cz),\]
for any $c\in \C^*$ of order $n$, and any $a\in \C^*$ such that $a^2=c^{\deg(p)}$.
\item\label{H1A6}
If $H_p$ is of even order $n\ge 2$ and $\deg(p)$ is even, then $H^1(A_\C(p))$ contains exactly six elements, namely the classes of
\[(x,y,z),(ax,ay,cz),(y,x,z),(-y,-x,z),( ay,a x,cz),( -ay,-a x,cz),\]
for any $c\in \C^*$ of order $n$, and any $a\in \C^*$ such that $a^2=c^{\deg(p)}$.
\end{enumerate}
\end{lemma}

\begin{proof}As $p$ is in reduced form, every element of $H_p$ is of the form $(cz)$ for some $c\in \C^*$ thanks to Lemma~\ref{lemma:Hp}.
Since $\deg(p)\ge 3$, we have by Lemma~\ref{AcapBdeg3} $A_\C(p)=(A_\C(p) \cap B_\C(p)) \rtimes \langle (y,x,z) \rangle$, with $A_\C(p) \cap B_\C(p)=\{  (ax,by,cz) \mid a,b,c\in \C^*, abp(z)=p(cz)\}$. Thus, we can define a surjective group homomorphism $\varphi\colon A_\C(p)\twoheadrightarrow   H_p\times  \langle (y,x,z) \rangle$ by sending $ (ax,by,cz)$ onto $(cz,\mathrm{id})$ and $(y,x,z) $ onto $(\mathrm{id},(y,x,z))$.

There are two cases to distinguish, both following from Lemma~\ref{lemma: H^1(H_p)}:
\begin{enumerate}[label=(\roman*), leftmargin=*]
\item\label{iHptrivial}
 If $H_p$ is infinite or finite of odd order, then $H^1(H_p)=\{1\}$.
\item\label{iiHpnontrivial}
 If $H_p$ is finite of even order $n\ge 2$, then $H^1(H_p)$ contains exactly two classes, namely the class of the identity and a second class that contains $(cz)$ for each $c\in \C^*$ of order $n$. 
\end{enumerate}
In case~\ref{iiHpnontrivial}, we fix $c\in \C^*$ of order $n$.

For each $1$-cocycle $\tau\in Z^1(A_\C(p))$, we may assume that $\sigma=\varphi(\tau)$ belongs to $\{\mathrm{id}\}\times  \langle (y,x,z) \rangle$ in Case~\ref{iHptrivial} and to $\{\mathrm{id}, (cz)\}\times  \langle (y,x,z) \rangle$ in Case~\ref{iiHpnontrivial}. This gives two or four possibilities for $\sigma$, respectively. Moreover, two $1$-cocycles that get mapped to different elements in $H_p \times \langle (y,x,z) \rangle$ cannot be equivalent. So, we may study the different possibilities for $\sigma$ separately.

We consider first the case where $\tau\in Z^1(A_\C(p))$ with $\sigma=\varphi(\tau)=(\mathrm{id},\mathrm{id})$. Then, $\tau=(ax,\frac{1}{a}y,z)$ for some $a\in \C^*$ with $a\overline{a}=1$. Choosing $\lambda\in \C$ with $\lambda^2= a$ and defining $\theta=(\lambda x,\frac{1}{\lambda} y,z)\in A_\C(p)$, we obtain $\theta^{-1}\circ\tau\circ\overline{\theta}=(\frac{a}{\lambda^2}x,\frac{\lambda^2}{a}y,z)=(x,y,z)$, since $\lambda\overline{\lambda}=1$.

Now, consider the case where $\tau\in Z^1(A_\C(p))$ with $\sigma=((cz),\mathrm{id})$. Then, $\tau=(ax,by,cz)$ for some $a,b\in \C^*$ with $a \overline{a}=b\overline{b}=1$ and $abp(z)=p(cz)=c^{\deg(p)}p(z)$. Let $\lambda\in \C$ with $\lambda\overline{\lambda}=1$ and define $\theta=(\lambda x,\frac{1}{\lambda} y,z)$. Then, $\theta^{-1}\circ\tau\circ\overline{\theta}=(\frac{a}{\lambda^2}x,\lambda^2by,cz)$. Choosing $\lambda$ with $\lambda^4=\frac{a}{b}$, we may thus assume that $b=a$, i.e., that $\tau=(ax,ay,cz)$ with $a^2=c^{\deg(p)}$. Repeating the same argument with $\lambda=\I$, we see that the two $1$-cocycles $(ax,ay,cz)$ and $(-ax,-ay,cz)$ are equivalent. Hence, there is only one class of $1$-cocycles associated to $\sigma=((cz),\mathrm{id})$.
 
Finally, we consider the case where $\tau\in Z^1(A_\C(p))$ with $\sigma=(\mathrm{id},(y,x,z))$ or $\sigma=((cz),(y,x,z))$. Then, $\tau=(ay,\frac{1}{\overline{a}}x,\mu z)$ for some $a\in \C^*$ satisfying $a\cdot \frac{1}{\overline{a}}=\mu^{\deg(p)}$, where $\mu=1$ or $\mu=c$. Choosing $\lambda\in \R_{>0}$ with $\lambda^2=\lvert a\rvert$ and defining $\theta=(\lambda x,\frac{1}{\lambda} y,z)$, we obtain $\theta^{-1}\circ\tau\circ\overline{\theta}=(\frac{a}{\lambda^2}y,\frac{\lambda^2}{\overline{a}}x,\mu z)$. So, we may assume that $\lvert a\rvert=1$, hence that $\tau=(ay,ax,\mu z)$. 

As $p(\mu z)=a^2p(z)$, we get $\mu^{\deg(p)}=a^2$. In particular, $a=\pm 1$ if $\mu=1$. To conclude the proof, it only remains to prove that the $1$-cocycles $(ay,ax,\mu z)$ and $(-ay,-ax,\mu z)$ are equivalent if and only if the following holds:
\begin{equation}\label{Condition}\tag{$\clubsuit$} \deg(p)\text{ is odd and } H_p\text{ is either infinite or  finite of even order.}\end{equation}

Suppose first that \eqref{Condition} holds. In this case, $(-z)\in H_p$, and $p(-z)=(-1)^{\deg(p)}p(z)=- p(z)$. Hence, $\theta=(x,-y,-z)\in A_\C(p)\cap B_\C(p)$, and $\theta^{-1}\circ (ay,ax,\mu z) \circ \overline{\theta}=(-ay,-ax,\mu z)$.

Suppose now that \eqref{Condition} does not hold, and suppose, by contradiction, that $\theta^{-1}\circ (ay,ax,\mu z) \circ \overline{\theta}=(-ay,-ax,\mu z)$ for some $\theta\in A_\C(p)$.

If $\theta\in A_\C(p)\cap B_\C(p)$, then $\theta=(\alpha x,\beta y, \gamma z)$ for some $\alpha,\beta,\gamma\in \C^*$ such that $\alpha\beta p(z)=p(\gamma z)$. This implies that $\theta^{-1}\circ(ay,ax,\mu z)\circ \overline{\theta}=(\frac{\overline{\beta}}{\alpha}ay,\frac{\overline{\alpha}}{\beta}ax,\frac{\overline{\gamma}}{\gamma}\mu z)$. Hence, $\beta=- \overline{\alpha}$ and $\gamma\in \R$. In particular, we have that $\alpha\beta=-\alpha\overline{\alpha}\in \R_{<0}$. Since $\alpha\beta p(z)=p(\gamma z)=\gamma^{\deg(p)}p(z)$, we also have $\alpha\beta=\gamma^{\deg(p)}$. This implies that $\deg(p)$ is odd and $\gamma<0$. As we assumed that \eqref{Condition} does not hold, $H_p$ is finite of odd order. But then, $(\gamma z)\notin H_p$. Contradiction.

If $\theta\in A_\C(p)\setminus B_\C(p)$, then write $\theta=(y,x,z)\circ \theta'$ with $\theta'\in A_\C(p)\cap B_\C(p)$. Since $(y,x,z)$ commutes with $\tau=(ay,ax,\mu z)$, the equality $\theta^{-1}\circ\tau\circ \overline{\theta}=\theta'^{-1}\circ\tau\circ\overline{\theta'}$ holds and we get a contradiction as above.
\end{proof}

\subsection{Real forms}
\begin{proposition}\label{Prop:RealformsDanielewski2a} Let $p=z^2-1$. The complex  surface \[D_{p}=\Spec(\C[x,y,z]/(xy-z^2+1))\] has exactly four nonisomorphic classes of real forms, which are those of the four real surfaces
\begin{align*}S_1&=\Spec(\R[x,y,z]/(x^2+y^2+z^2+1))\\
S_2&=\Spec(\R[x,y,z]/(x^2+y^2+z^2-1))\\
S_3&=\Spec(\R[x,y,z]/(x^2-y^2+z^2-1))\\
S_4&=\Spec(\R[x,y,z]/(x^2-y^2+z^2+1))
\end{align*}
All four are pairwise not homeomorphic: their real parts are diffeomorphic to
\[S_1(\mathbb{R})=\varnothing, \, S_2(\mathbb{R})\simeq\mathbb{S}^2, \, S_3(\mathbb{R})\simeq \R^2 \setminus \{(0,0)\}, \, S_4(\mathbb{R})\simeq \mathbb{R}^2 \amalg \mathbb{R}^2.\]
\end{proposition}

\begin{proof}
By Lemma~\ref{Lemm:H1B-new} and Lemma~\ref{Az21}, $H^1(\Aut_\C(D_{p}))$ contains exactly four elements, namely the classes of $\tau_3=(x,y,z)$, $\tau_4=(-x,-y,-z)$, $\tau_1=(y,x,-z)$, $\tau_2=(-y,-x,z)$. Therefore, there are exactly four nonisomorphic real forms of $D_p$.

To see that they correspond to the real surfaces $S_1,\ldots,S_4$, we produce, for every $i=1,2,3,4$, an element $\theta_i\in \mathrm{GL}_3(\C)\subset \Aut(\A^3_\C)$ such that $\tau_i \circ \rho =\theta_i \circ \rho\circ\theta_i^{-1} $, where $\rho$ is the standard real form $(x,y,z)\mapsto (\overline{x},\overline{y},\overline{z})$ on $\A^3_\C$, and such that $\theta_i^{-1}(D_p)$ is the complexification of $S_i$, i.e.,~is $S_i\times _{\mathrm{Spec}(\mathbb{R})}\mathrm{Spec}(\mathbb{C})$. 
\[\begin{array}{|c|c|c|c|}
\hline
i &\tau_i&\theta_i \text{ with } \tau_i \circ \rho \circ \theta_i =\theta_i \circ \rho & \theta_i^*(xy-z^2+1)\\
\hline
1 & (y,x,-z) & (x+\mathbf{i}y,x-\I y,\I z)  & x^2+y^2+z^2+1 \\

2 &(-y,-x,z) & (x+\mathbf{i}y,-x+\mathbf{i}y,z) &     -(x^2+y^2+z^2-1) \\

3&(x,y,z)& (x+y,y-x,z) &  -(x^2-y^2+z^2-1) \\

4& (-x,-y,-z)& (\I(-x+y),\I(x+y),\I z)   &  x^2-y^2+z^2+1\\ \hline \end{array}\]

From the equations of $S_1$ and $S_2$, we see that $S_1(\R)=\varnothing$ and $S_2(\R)=\mathbb{S}^2$. The map $(x,y,z) \mapsto (y, \tfrac{x}{\sqrt{x^2+z^2}}, \tfrac{z}{\sqrt{x^2+z^2}})$ provides an explicit diffeomorphism from $S_3(\R)$ to the cylinder $\R \times \mathbb{S}^1$, which is diffeomorphic to the punctured plane $\R^2 \setminus \{(0,0)\}$. For $S_4(\R)$, note that $x^2+z^2=y^2-1$ implies that $y\neq 0$. Then, $S_4(\R)=\{(x,y,z) \mid y>0 \} \amalg \{(x,y,z) \mid y<0 \}$ is diffeomorphic to the disjoint union of two copies of $\R^2$. 
\end{proof}

\begin{proposition} \label{Prop:RealformsDanielewski2b}
Let $p=z^2$. The complex  surface \[D_{p}=\Spec(\C[x,y,z]/(xy-z^2))\] has exactly two nonisomorphic classes of real forms, which are those of the two real surfaces
\begin{align*}
T_1&=\Spec(\R[x,y,z]/(x^2+y^2+z^2)),\\
T_2&=\Spec(\R[x,y,z]/(x^2-y^2-z^2)).
\end{align*}
Both are pairwise not homeomorphic: $T_1(\mathbb{R})$ consists of only one point, while $T_2(\mathbb{R})$ is infinite; it is a cone over $\mathbb{S}^1$.
\end{proposition}

\begin{proof}
By Lemma~\ref{Lemm:H1B-new} and Lemma~\ref{Az2}, $H^1(\Aut_\C(D_{p}))$ contains exactly two elements, namely the classes of $\tau_2=(x,y,z)$, $\tau_1=(y,x,-z)$. Therefore, there are exactly two nonisomorphic real forms of $D_p$.

To see that they correspond to the real surfaces $T_1,T_2$, we give, for every $i=1,2$, an element $\theta_i\in \mathrm{GL}_3(\C)\subset \Aut(\A^3_\C)$ such that $\tau_i \circ \rho =\theta_i \circ \rho\circ\theta_i^{-1}$, where $\rho$ is the standard real form $(x,y,z)\mapsto (\overline{x},\overline{y},\overline{z})$ on $\A^3_\C$, and such that $\theta_i^{-1}(D_p)$ is the complexification of $T_i$, i.e.,~is $T_i\times _{\mathrm{Spec}(\mathbb{R})}\mathrm{Spec}(\mathbb{C})$. 
\[\begin{array}{|c|c|c|c|}
\hline
i &\tau_i&\theta_i \text{ with }  \tau_i \circ \rho \circ \theta_i =\theta_i \circ \rho& \theta_i^*(xy-z^2)\\
\hline
1 & (y,x,-z) & (x+\mathbf{i}y,x-\I y,\I z) &    x^2+y^2+z^2 \\

2&(x,y,z)& (x-y,x+y,z) & x^2-y^2-z^2 \\ \hline \end{array}\]
The equation of $T_1$ directly gives $T_1(\mathbb{R})=\{(0,0,0)\}$, whereas $T_2(\mathbb{R})$ is a cone over the conic $x^2-y^2=z^2$ in $\mathbb{P}^2_\R$, whose set of real points is diffeomorphic to $\mathbb{S}^1$.
\end{proof}

{\begin{proposition}\label{Prop:RealformsDanielewski3}
Let $p\in\R[z]$ be a polynomial of degree $d\geq3$ in reduced form and define $D_{p}=\Spec(\C[x,y,z]/(xy-p(z)))$.
\begin{enumerate}[leftmargin=*]
\item \label{RealFormDHpinfinite} If $H_p$ is infinite, then $p=z^d$ and there are two cases:
\begin{enumerate}[leftmargin=*]
\item \label{RealFormDHpinfinitedodd}If  $d$ is odd, then $D_p$ has exactly two isomorphism classes of real forms, namely those of
\[\Spec(\R[x,y,z]/(x^2\pm y^2-z^d))\]
\item \label{RealFormDHpinfinitedeven} If $d$ is even, then $D_p$ has exactly three isomorphism classes of real forms, namely those of
\begin{align*}&\Spec(\R[x,y,z]/(x^2+y^2+z^d)), \\
\text{and}\quad &\Spec(\R[x,y,z]/(x^2\pm y^2-z^d)).
\end{align*}
\end{enumerate}
\item \label{RealFormDHpfinite} If $H_p$ is cyclic of order $n$, then $p=z^mq(z^n)$ for some integer $m\geq0$ and some monic polynomial $q\in\R[z]\setminus \R$ with $q(0)\neq 0$, and there are three cases:
\begin{enumerate}[leftmargin=*]
\item\label{RealFormDHpfinitenodd} If  $n$ is odd, then $D_p$ has exactly three isomorphism classes of real forms, namely those of
\begin{align*}&\Spec(\R[x,y,z]/(x^2+y^2+z^mq(z^n))), \\
\text{and}\quad &\Spec(\R[x,y,z]/(x^2\pm y^2-z^mq(z^n))).
\end{align*}
\item\label{RealFormDHpfinitenevendodd} If $n$ is even and $\deg(p)$ -- and thus $m$ -- is odd, then $D_p$ has exactly four isomorphism classes of real forms, namely those of
\[\Spec(\R[x,y,z]/(x^2\pm y^2-z^mq(\pm z^n))).\]
\item\label{RealFormDHpfinitenevendeven6} If  $n$ is even and  $\deg(p)$ -- and thus $m$ -- is even, then $D_p$ has exactly six isomorphism classes of real forms, namely those of
\begin{align*}&\Spec(\R[x,y,z]/(x^2+y^2+z^mq(\pm z^n))), \\
\text{and}\quad &\Spec(\R[x,y,z]/(x^2\pm y^2-z^mq(\pm z^n))).
\end{align*}
\end{enumerate}
\end{enumerate}
\end{proposition}

\begin{proof}
Define 
\[\begin{array}{lll}
	\tau_1=(x,y,z), & \tau_2=(ax,ay,cz), & \tau_3=(y,x,z)\\
	\tau_4=(-y,-x,z), & \tau_5=(ay,a x,cz), & \tau_6=(-ay,-a x,cz),
\end{array} 
\]
which are the $1$-cocycles appearing in Lemma~\ref{Lemm:H1A}.

\ref{RealFormDHpinfinite} Suppose that $H_p$ is infinite. Then, $D_p$ is the surface of equation $xy=z^d$, and by Lemma~\ref{Lemm:H1A}\ref{H1A2}-\ref{H1A3}, we only need to consider $\tau_1$, $\tau_3$ and $\tau_4$.
In the table below, we produce, for every $i\in\{1,3,4\}$, an element $\theta_i\in \mathrm{GL}_3(\C)\subset \Aut(\A^3_\C)$ such that $\tau_i \circ \rho =\theta_i \circ \rho\circ\theta_i^{-1}$, where $\rho$ is the standard real form $(x,y,z)\mapsto (\overline{x},\overline{y},\overline{z})$ on $\A^3_\C$ and compute the equation of the hypersurface $\theta_i^{-1}(D_p)\subset\A_\C^3$. Combining Lemma~\ref{Lemm:H1B-new} with Lemma~\ref{Lemm:H1A}, this proves~\ref{RealFormDHpinfinite}.
\[ \begin{array}{|c|c|c|c|c|}
\hline
i &\tau_i&\theta_i \text{ with } \tau_i \circ \rho \circ \theta_i =\theta_i \circ \rho & \theta_i^*(xy-z^d)\\
\hline
1&(x,y,z)& (x+y,x-y,z) & x^2-y^2-z^d \\

3 & (y,x,z) &(x+\I y,x-\I y,z) &    x^2+y^2-z^d \\

4 & (-y,-x,z) &(x+\I y,-x+\I y,z) &    -(x^2+y^2+z^d) \\

\hline \end{array}\]

\ref{RealFormDHpfinite} Suppose that $H_p$ is cyclic of finite order $n\geq1$. Then, $D_p$ is given by an equation of the form $xy=z^mq(z^n)$ with $m\geq0$ and $\deg(q) \geq 1$ such that $q(0)\neq 0$.  Let $c=\e^{2\pi\I/n}$ be a primitive $n$-th root of unity and set $a=\e^{2\pi\I m/2n}$, which satisfies $a^2=c^m=c^{\deg(p)}$. 

Fix $\alpha=\e^{2\pi\I/2n}$ and $\beta=\e^{2\pi\I m/4n}$, for which $\alpha^2=c$ and $\alpha^n=-1$, and $\beta^2=a=\alpha^m$, respectively. In the table below, we produce, for every $i\in\{1,\ldots,6\}$, an element $\theta_i\in \mathrm{GL}_3(\C)\subset \Aut(\A^3_\C)$ such that $\tau_i \circ \rho =\theta_i \circ \rho\circ\theta_i^{-1}$, where $\rho$ is the standard real form $(x,y,z)\mapsto (\overline{x},\overline{y},\overline{z})$ on $\A^3_\C$ and compute the equation of the hypersurface $\theta_i^{-1}(D_p)\subset\A_\C^3$. Combining Lemma~\ref{Lemm:H1B-new} with Lemma~\ref{Lemm:H1A}, this proves~\ref{RealFormDHpfinite}. 

\[ \begin{array}{|c|c|c|c|}
\hline
i &\tau_i&\theta_i \text{ with }\tau_i \circ \rho \circ \theta_i =\theta_i \circ \rho& \theta_i^*(xy-z^mq(z^n))\\
\hline
1&(x,y,z)& (x+y,x-y,z) &  x^2-y^2-z^mq(z^n) \\

2&(ax,ay,cz)& (\beta(x+y),\beta(x-y),\alpha z) & \beta^2(x^2-y^2-z^mq(-z^n)) \\

3 & (y,x,z) &(x+\I y,x-\I y,z)   & x^2+y^2-z^mq(z^n) \\

4 & (-y,-x,z) &(x+\I y,-x+\I y,z)  & -(x^2+y^2+z^mq(z^n)) \\

5 & (ay,ax,cz) &(\beta(x+\I y),\beta(x-\I y),\alpha z) &    \beta^2(x^2+y^2-z^mq(-z^n)) \\

6 & (-ay,-ax,cz) &(\beta(x+\I y),\beta(-x+\I y),\alpha z)   & -\beta^2(x^2+y^2+z^mq(-z^n))\\
 \hline \end{array}\]
\end{proof}
We finalise this section by proving Theorem~\ref{ThmA} which summarises Propositions~\ref{Prop:RealformsDanielewski2a},~\ref{Prop:RealformsDanielewski2b} and~\ref{Prop:RealformsDanielewski3}.
\begin{proof}[Proof of Theorem~$\ref{ThmA}$]
We recall that $p\in \R[z]$ is a polynomial in reduced form of degree $d\ge 2$, $p(z)=z^mq(z^n)$ where $m\ge 0$, $n\ge 1$, $q\in \R[z]$, $q(0)\neq 0$ and where $q$ and $n$ are chosen such that $n$ is maximal if $q\neq1$. In particular, $q$, $n$ and $m$ are uniquely determined by $p$. 
	
We first remark that $S_{abc}$ is a real form of $D_p$ for all $a,b,c\in \{0,1\}$. Indeed, the linear map $(x+\I^{a-1}y,x-\I^{a-1}y,z)\in\Aut(\A^3_\C)$ sends the hypersurface $x^2+(-1)^ay^2+(-1)^bz^mq((-1)^cz^n)=0$ onto that of equation $xy+(-1)^bz^mq((-1)^cz^n)=0$, which is isomorphic to $D_p$ by Theorem~\ref{theo:DanielewskiIsoClasses}. Propositions~\ref{Prop:RealformsDanielewski2a},~\ref{Prop:RealformsDanielewski2b} and~\ref{Prop:RealformsDanielewski3} then give the number $2\le i\le 6$ of isomorphism classes together with a list of representatives.

Suppose first that $q=1$. Then $p(z)=z^d=z^m$ and  $H_p$ is thus infinite. 
If $d=2$, then Proposition~\ref{Prop:RealformsDanielewski2b} gives $i=2$ together with the representatives $S_{000}$ and $S_{110}$. 
If $d\ge 3$, Proposition~\ref{Prop:RealformsDanielewski3}\ref{RealFormDHpinfinite} gives $i=2$ when $d$ is odd and $i=3$ when $d$ is even. In the case where $d$ is odd,  Proposition~\ref{Prop:RealformsDanielewski3}\ref{RealFormDHpinfinite}\ref{RealFormDHpinfinitedodd} gives the two representatives $\Spec(\R[x,y,z]/(x^2+ y^2-z^d))=S_{010}$ and $\Spec(\R[x,y,z]/(x^2- y^2-z^d))=S_{110}$.  Using the isomorphism $(x,y,-z)\colon S_{000}\iso S_{010}$, we obtain the two representatives given in the statement of Theorem~\ref{ThmA}. In the  case where $d$ is even, the three representatives of Proposition~\ref{Prop:RealformsDanielewski3}\ref{RealFormDHpinfinite}\ref{RealFormDHpinfinitedeven} are precisely $S_{000}$,  $S_{010}$ and $S_{110}$.

Suppose now that $q\neq 1$. Hence, $\deg(q)\ge 1$, and as $n$ was chosen maximal, the group \[H_p=\{\lambda\in \C^*\mid p(\lambda z)=\lambda^dp(z)\}=\{(\lambda z)\mid \lambda\in \C^*, \lambda^n=1\}\] is cyclic of order $n$ by Lemma~\ref{lemma:Hp}. If $d=2$, then $p=z^2+\mu$, for some $\mu\in \R^*$. So $(m,n)=(0,2)$ and the surface $D_p$ is  isomorphic to $D_{p'}$ with $p'=z^2-1$ by Theorem~\ref{theo:DanielewskiIsoClasses}. Hence, Proposition~\ref{Prop:RealformsDanielewski2a}  gives $i=4$ and provides the four representatives  $\Spec(\R[x,y,z]/(x^2\pm y^2+z^2\pm 1))$. We now need to check that these surfaces are isomorphic to the four surfaces $S_{abb},a,b\in \{0,1\}$ that are given in the statement of Theorem~\ref{ThmA}. Since these latter are defined by $\Spec(\R[x,y,z]/(x^2\pm y^2+z^2\pm \mu))$, it actually suffices to apply the linear automorphism $(\xi x, \xi y, \xi z)\in\Aut(\A^3_\R)$ where $\xi=\sqrt{\lvert \mu \rvert}$.

If $d\ge 3$, Proposition~\ref{Prop:RealformsDanielewski3}\ref{RealFormDHpfinite}  specifies  three different cases.

If $n$ is odd, then $i=3$ and the representatives in Proposition~\ref{Prop:RealformsDanielewski3}\ref{RealFormDHpfinite}\ref{RealFormDHpfinitenodd} are precisely the surfaces $S_{000}$, $S_{010}$ and $S_{110}$.

If $n$ is even and $d$ is odd, then $i=4$ and the representatives  given by Proposition~\ref{Prop:RealformsDanielewski3}\ref{RealFormDHpfinite}\ref{RealFormDHpfinitenevendodd}  are the surfaces $S_{a1c}$ with $a,c\in \{0,1\}$. As the map $(x,y,-z)\in\Aut(\A^3_\R)$ sends $S_{a1c}$ to $S_{a0c}$,  we obtain $S_{a1c}\simeq S_{a0c}$, and in particular $S_{a1c}\simeq S_{acc}$. This gives the result.

The remaining case is when $n$ and $d$ are both even. Here, $i=6$ and  the real forms are $S_{00c},S_{a1c},a,c\in \{0,1\}$ by Proposition~\ref{Prop:RealformsDanielewski3}\ref{RealFormDHpfinite}\ref{RealFormDHpfinitenevendeven6}.
 \end{proof}

\section{The surfaces $(\A^1_\C\setminus \{0\})^2$ and $\A^1_\C\times (\A^1_\C\setminus \{0\})$}\label{CCC}

In this section, we compute the real forms of the two affine surfaces $(\A^1_\C\setminus \{0\})^2$ and $\A^1_\C\times (\A^1_\C\setminus \{0\})$.  In Propositions~\ref{Prop:Cstar2realforms} and \ref{Prop:CCstarrealforms}, we prove that these surfaces have respectively six and four isomorphism classes of real forms. In the case of $(\A^1_\C\setminus \{0\})^2$, a partial result, together with a sketch of the proof, is given in \cite[Lemma~1.5 and Remark 1.6]{MoserTerpereau}. Our proof follows essentially the same lines.

The following well-known result is an easy exercise. We give the proof for the sake of completeness.

\begin{lemma}\label{InvGL2Z}
There are exactly three conjugacy classes of elements of order $2$ in $\mathrm{GL}_2(\Z)$, namely those of $\sigma_1=\left(\begin{smallmatrix} 1 & \phantom{-}0\\ 0 & -1 \end{smallmatrix}\right)$, $\sigma_2=\left(\begin{smallmatrix} -1 & \phantom{-}0\\ \phantom{-}0 & -1 \end{smallmatrix}\right)$ and $\sigma_3=\left(\begin{smallmatrix} 0 & 1 \\ 1 & 0 \end{smallmatrix}\right)$. 
\end{lemma}
\begin{proof}We first prove that the involutions $\sigma_1,\sigma_2,\sigma_3$ are pairwise not conjugate. As $\det(\sigma_2)=1$ and $\det(\sigma_1)=\det(\sigma_3)=-1$, we only need to prove that $\sigma_1$ and $\sigma_3$ are not conjugate. If they were, we would have a matrix $M=\left(\begin{smallmatrix} a & b\\ c & d \end{smallmatrix}\right)\in \mathrm{GL}_2(\Z)$ such that $\left(\begin{smallmatrix} a \\ c \end{smallmatrix}\right)$ and $\left(\begin{smallmatrix} b \\ d \end{smallmatrix}\right)$ are eigenvectors of $\sigma_3$ of eigenvalue $1$ and $-1$ respectively. This would imply  $c=a$ and $d=-b$, which is impossible, as $\det(M)=ad-bc=-2ab\notin\{\pm1 \}$.

It remains to prove that every element $M\in \mathrm{GL}_2(\Z)$ of order $2$ is conjugate to $\sigma_1,\sigma_2$ or $\sigma_3$. If $M\neq\sigma_2$, then the eigenvalues of $M$ are $1$ and $-1$. Consider an eigenvector of $M$ with integer entries prime to each other and complete it to a matrix of $\mathrm{GL}_2(\Z)$ that conjugates $M$ to $M'=\left(\begin{smallmatrix} 1 & b\\ 0 & d \end{smallmatrix}\right)$ for some $b,d\in \Z$. Note that $d=-1$, since $M$ has eigenvalues $1$ and $-1$. Conjugating $M'$ by  $\left(\begin{smallmatrix} 1 & \mu\\ 0 & 1 \end{smallmatrix}\right)$ with $\mu\in \Z$, we get the matrix $\left(\begin{smallmatrix} 1 & b-2\mu\\ 0 & -1 \end{smallmatrix}\right)$. If $b$ is even, then $M$ is thus  conjugate to $\sigma_1$. If $b$ is odd, then $M$ is conjugate to $\left(\begin{smallmatrix} 1 & \phantom{-}1\\ 0 & -1 \end{smallmatrix}\right)$, which is conjugate to $\sigma_3$ by $\left(\begin{smallmatrix} 1 & 1\\ 1 & 0 \end{smallmatrix}\right)$.
\end{proof}

\begin{proposition}\label{Prop:Cstar2realforms}$ $
\begin{enumerate}[leftmargin=*]
\item\label{realformsCstar}
The affine complex curve $\A^1_\C\setminus \{0\}$ has exactly three equivalence classes of real structures, namely those of \[\rho_1\colon x\mapsto \overline{x}, \, \rho_2\colon x\mapsto \overline{x}^{-1}, \, \rho_3\colon x\mapsto -\overline{x}^{-1}.\] The corresponding real forms of $\A^1_\C\setminus \{0\}$ are the three affine conics $\Gamma_1,\Gamma_2,\Gamma_3\subset \A^2_\R$ given  by 
\[xy-1=0, \, x^2+y^2-1=0 \text{ and } x^2+y^2+1=0,\]
whose real parts are diffeomorphic to $\Gamma_1(\R)\simeq \R^*$, $\Gamma_2(\R)\simeq \mathbb{S}^1$ and $\Gamma_3(\R)=\varnothing$, respectively.
\item\label{realformsCstar2}
The affine complex surface $(\A^1_\C\setminus \{0\})^2$ has exactly six isomorphism classes of real forms, namely those of
\[ \Gamma_1 \times \Gamma_1, \Gamma_1\times \Gamma_2, \Gamma_1\times \Gamma_3, \Gamma_2\times \Gamma_2, \Gamma_3\times \Gamma_3   \text{ and }\A^2_{\R}\setminus \{x^2+y^2=0\}.\]
$($Note that $\Gamma_2\times \Gamma_3$ is isomorphic to $\Gamma_3\times \Gamma_3.)$
\end{enumerate}\end{proposition}
\begin{proof}
We recall that  for $n\ge 1$, the   invertible regular functions on $(\A^1_\C\setminus \{0\})^n$ are  the Laurent monomials $\mu x_1^{a_1}\cdots x_n^{a_n}$, with $\mu\in \C^*$, $a_1,\ldots,a_n\in \Z$. This implies that $\Aut_\C((\A^1_\C\setminus \{0\})^n)\simeq (\C^*)^n\rtimes \GL_n(\Z)$, and gives in particular 
\begin{align*}
\Aut(\A^1_\C\setminus \{0\})&=\{\lambda x^a\mid \lambda \in \C^*, a=\pm 1\},\\
\Aut((\A^1_\C\setminus \{0\})^2)&=\{(a x^{m_{11}}y^{m_{12}},b x^{m_{21}}y^{m_{22}})\mid a,b\in \C^*, \left(\begin{smallmatrix} m_{11} & m_{12}\\ m_{21} & m_{22} \end{smallmatrix}\right)\in \mathrm{GL}_2(\Z)\}.\end{align*}

We  prove~\ref{realformsCstar}. As the complexification of $\Gamma_i$ is a smooth affine conic with two points at infinity, it is isomorphic to $\A^1_\C\setminus \{0\}$, and thus, $\Gamma_i$ is a real form of $\A^1_\C\setminus \{0\}$. Since $\Gamma_1$, $\Gamma_2$ and $\Gamma_3$ have nonhomeomorphic real parts, we get three pairwise nonisomorphic real forms. We now prove that these are the only ones. We fix the standard real structure $\rho_1$ that corresponds to the real form $\A^1_\R\setminus \{0\}$, isomorphic to $\Gamma_1$. The description of $\Aut(\A^1_\C\setminus \{0\})$ implies that every element of $Z^1(\Aut(\A^1_\C\setminus \{0\}))$ is either of the form $\nu=( \mu x)$ with $\mu\in \C^*$, $\mu\overline{\mu}=1$, or of the form $\nu=(\mu x^{-1})$ with $\mu\in \R^*$. In the first case, we reduce to $\mu=1$, as $H^1(\C^*)=\{1\}$ (Lemma~\ref{Lem:H1CCstar}), and obtain the trivial real form $\Gamma_1$. In the second case, we choose $\alpha=(\lambda x)$ with $\lambda \in \R,$  $\lambda^2=\lvert \mu\rvert$ and obtain $\alpha^{-1}\circ\nu \circ \overline{\alpha}=(\pm x^{-1})$. This gives the two real structures $\rho_2$ and $\rho_3$, which then necessarily correspond to $\Gamma_2$ and $\Gamma_3$. As $\Gamma_3(\R)=\varnothing$ and as no $x\in \C^*$ satisfies $x=\rho_3(x)=-\overline{x}^{-1}$, we find that $\rho_i$ corresponds to $\Gamma_i$ for $i=1,2,3$.

It remains to prove~\ref{realformsCstar2}. We fix the standard real form $\rho_1\times \rho_1$ and compute $H^1(\Aut((\A^1_\C\setminus \{0\})^2))$. Let $\nu\in Z^1(\Aut((\A^1_\C\setminus \{0\})^2))$ be a $1$-cocycle. As $\Aut((\A^1_\C\setminus \{0\})^2)\simeq (\C^*\times \C^*)\rtimes \GL_2(\Z)$, the $1$-cocycle $\nu$ gives rise to an involution $\sigma\in \GL_2(\Z)$. Up to conjugation, $\sigma$ is equal to precisely one of $\sigma_0=\left(\begin{smallmatrix} 1 & 0\\ 0 & 1 \end{smallmatrix}\right)$, $\sigma_1=\left(\begin{smallmatrix} 1 & \phantom{-}0\\ 0 & -1 \end{smallmatrix}\right)$, $\sigma_2=\left(\begin{smallmatrix} -1 & \phantom{-}0\\ \phantom{-}0 & -1 \end{smallmatrix}\right)$, or $\sigma_3=\left(\begin{smallmatrix} 0 & 1 \\ 1 & 0 \end{smallmatrix}\right)$, see Lemma~\ref{InvGL2Z}. These four being pairwise not conjugate in $\GL_2(\Z)$, two $1$-cocycles arising from two different $\sigma_i$, $\sigma_j$ are not equivalent, so we can study each $\sigma_i$ separately. For~$\sigma_0$, $\sigma_1$ and $\sigma_2$, we can, on each component of the map $\nu$, apply the same reduction as we did above for $\Aut(\A_\C^1\setminus \{0\})$.

If $\sigma=\sigma_0$, then $\nu=(\lambda x,\mu y )$, where $\lambda,\mu\in \C^*$ have modulus $1$. As $H^1(\C^*)=\{1\}$, we can reduce to the case $\lambda=\mu=1$, and get the real structure $\rho_1\times \rho_1$, and thus the real form $\Gamma_1\times \Gamma_1\simeq (\A^1_\R\setminus \{0\})^2$.

If $\sigma=\sigma_1$, then  $\nu=(\lambda x,\mu y^{-1} )$, where $\lambda\in \C^*$ has modulus $1$ and $\mu\in \R^*$. We reduce to $\lambda=1$ and $\mu=\pm 1$,  get two real structures $\rho_1\times \rho_2$ and $\rho_1\times \rho_3$, and thus the real forms $\Gamma_1\times \Gamma_2$ and $\Gamma_1\times \Gamma_3$. These real forms are not isomorphic, as the second one has no real points, whereas the first has.

If $\sigma=\sigma_2$, then $\nu=(\lambda x^{-1} ,\mu y^{-1} )$, where $\lambda,\mu\in \R^*$. We reduce to $\lambda,\mu\in \{\pm 1\}$,  get the four real structures $\rho_i\times \rho_j$, where $i,j=2,3$, and hence four real forms $\Gamma_i\times \Gamma_j$. With $\alpha=(x,xy)$, we obtain $\alpha^{-1} \circ (-x^{-1},y^{-1}) \circ \overline{\alpha}=(-x^{-1},-y^{-1})$, hence an isomorphism $\Gamma_2\times \Gamma_3\iso \Gamma_3\times \Gamma_3$. Similarly, $\alpha=(y,x)$ provides an isomorphism  $\Gamma_2\times \Gamma_3\iso \Gamma_3\times \Gamma_2$. As $\Gamma_2\times \Gamma_2$ has real points and $\Gamma_2\times \Gamma_3$ does not, we obtain exactly two isomorphism classes of real forms in this case.

If $\sigma=\sigma_3$, then $\nu=(\frac{1}{\lambda}y,\overline{\lambda}x)$, for some $\lambda\in \C^*$. With $\alpha=(\frac{1}{\lambda} x,y)$, we obtain $\alpha^{-1} \circ \nu \circ \overline{\alpha}=(y,x)$, resulting in the real structure $\rho'\colon (x,y)\mapsto (\overline{y},\overline{x})$. We use the isomorphism $(\A^1_\C\setminus \{0\})^2=\A^2_\C\setminus\{xy=0\}\iso \A^2_\C\setminus \{x^2+y^2=0\}$, $(x,y)\mapsto (x+y,\I (x-y))$. It conjugates the real structure $\rho'$ to the standard real structure $(x,y)\mapsto (\overline{x},\overline{y})$. The real form induced is then isomorphic to $\A^2_\R\setminus\{x^2+y^2=0\}$.
\end{proof}
\begin{proposition}\label{Prop:CCstarrealforms}$ $
The affine complex surface $\A^1_\C\times (\A^1_\C\setminus \{0\})$ has exactly four isomorphism classes of real forms, namely those of
\[ \A^1_\R \times \Gamma_1, \, \A^1_\R \times \Gamma_2, \, \A^1_\R \times \Gamma_3 \text{ and } \mathbb{P}^2_{\R}\setminus \{x^2+y^2=0\},\]
where $\Gamma_1,\Gamma_2$ and $\Gamma_3$ are the real forms of $\A^1_\C\setminus \{0\}$, given in Proposition~$\ref{Prop:Cstar2realforms}\ref{realformsCstar}.$
\end{proposition}
\begin{proof}
First, recall that $\Aut(\A^1_\C\times (\A^1_\C\setminus \{0\}))$ is equal to 
\[\{(\lambda xy^m +c(y),\mu y^{\pm 1})\mid \lambda,\mu\in \C^*, m\in \Z, c\in \C[y,y^{-1}]\subset \C(y)\}.\]
To obtain this, we can use the fact that every morphism $\A^1_\C\to\A^1_\C\setminus \{0\}$ is constant, so any automorphism $\varphi$ of $\A^1_\C\times (\A^1_\C\setminus \{0\})$ sends a fibre of the first projection to another fibre. Thus, $\varphi$ is of the form $(a(x,y),b(y))$, where $x\mapsto a(x,y)$ is an automorphism of $\A_\C^1$ for every $y$, and where $y \mapsto b(y)$ is an automorphism of $\A_\C^1 \setminus \{0\}$, since the inverse of $\varphi$ is of the same form. 

We fix the standard real structure $(x,y)\mapsto (\overline{x},\overline{y})$ on $\A^1_\C\times (\A^1_\C\setminus \{0\})$, corresponding to the real form $\A^1_\R \times (\A^1_\R\setminus \{0\})\simeq \A^1_\R \times \Gamma_1$, see Proposition~\ref{Prop:Cstar2realforms}\ref{realformsCstar}, and compute $H^1(\Aut(\A^1_\C\times (\A^1_\C\setminus \{0\})))$. 
We consider the group homomorphism $\theta\colon\Aut(\A^1_\C\times (\A^1_\C\setminus \{0\}))\to \GL_2(\Z)$ that sends $(\lambda xy^m+c(y),\mu y^{\pm 1})$ onto $\left(\begin{smallmatrix} 1 & m \\ 0 & \pm 1 \end{smallmatrix}\right)$. 

Let $\nu\in Z^1(\Aut(\A^1_\C\times (\A^1_\C\setminus \{0\})))$ be a $1$-cocycle. Then, the matrix $\theta(\nu)$ is  an involution in the group $H=\{\left(\begin{smallmatrix} 1 & m \\ 0 & \pm 1 \end{smallmatrix}\right)\mid m\in \Z\,\}\subset \GL_2(\Z)$. This involution is either $\sigma_0=\left(\begin{smallmatrix} 1 & 0 \\ 0 & 1 \end{smallmatrix}\right)$, $\sigma_1=\left(\begin{smallmatrix} 1 & \phantom{-}0 \\ 0 & -1 \end{smallmatrix}\right)$, $\sigma_2=\left(\begin{smallmatrix} 1 & -1 \\ 0 & -1 \end{smallmatrix}\right)$, or more generally $\left(\begin{smallmatrix} 1 & m \\ 0 & -1 \end{smallmatrix}\right)$ for any $m\in \Z$. Conjugating the latter by $\left(\begin{smallmatrix} 1 & a \\ 0 & 1 \end{smallmatrix}\right)$ gives the matrix $\left(\begin{smallmatrix} 1 & m-2a \\ 0 & -1 \end{smallmatrix}\right)$, so we may reduce to the cases of $\sigma_0$, $\sigma_1$ or $\sigma_2$. Since $\left(\begin{smallmatrix} 0 & 1 \\ 1 & 0 \end{smallmatrix}\right)$ is conjugate to $\sigma_2$, using $\left(\begin{smallmatrix} 0 & -1 \\ 1 & -1 \end{smallmatrix}\right)$, Lemma~\ref{InvGL2Z} implies that the involutions $\sigma_1,\sigma_2$ are not conjugate in $\GL_2(\Z)$, and thus also not conjugate in $H$. We then obtain three disjoint families of real forms, up to isomorphism, and may consider the three cases separately.

Consider first the case where $\theta(\nu)=$ $\sigma_0$. Thus, $\nu=(\lambda x+c(y),\mu y)$ for some $\lambda,\mu\in \C^*$ of modulus $1$ and $c\in \C[y, y^{- 1}]$. Considering $\alpha^{-1}\circ\nu\circ \overline{\alpha}$ with $\alpha=(\xi_1x,\xi_2 y)$ where $\xi_1^2=\lambda,\xi_2^2=\mu$, we may reduce to the case where $\lambda=\mu=1$. Then, the $1$-cocycle condition  $\nu \circ\overline{\nu}=1$ gives $c(y)+\overline{c}(y)=0$. Considering $\alpha^{-1}\circ\nu\circ \overline{\alpha}$ with $\alpha=(x+c(y)/2,y)$, we further reduce to the trivial real structure, corresponding to the real form $\A^1_\R\times (\A^1_\R\setminus \{0\})\simeq \A^1_\R\times \Gamma_1$.

We now consider the case where $\theta(\nu)=\sigma_1$. Thus,  $\nu=(\lambda x +c(y),\mu y^{-1})$ for some $\lambda\in \C^*$ with $\lvert \lambda\rvert=1$, $\mu\in \R^*$ and $c\in \C[y, y^{- 1}]$. Considering $\alpha^{-1}\circ\nu\circ \overline{\alpha}$ with $\alpha=(\xi_1 x,\xi_2 y)$, $\xi_1\in \C^*,\xi_2\in \R^*$, $\xi_1^2=\lambda$, $\xi_2^2=\lvert \mu\rvert$, we reduce to the case where $\lambda=1, \mu\in \{\pm 1\}$.  Then, the $1$-cocycle condition  $\nu \circ\overline{\nu}=1$ gives $\overline{c}(y)+c(\mu y^{-1})=0$. Considering $\alpha^{-1}\circ\nu\circ \overline{\alpha}$ with $\alpha=(x-\overline{c}(y)/2,y)$, we reduce to $c=0$. This gives the two real structures $(x,y)\mapsto ( \overline{x},\overline{y}^{-1})$ and $(x,y)\mapsto ( \overline{x},-\overline{y}^{-1})$ and the real forms $\A^1_\R\times \Gamma_2$ and $\A^1_\R\times \Gamma_3$. The first one has real points and the second does not, so these are not isomorphic.

We now study the case where $\theta(\nu)=\sigma_2$. Thus,  $\nu=(\lambda xy^{-1}+c(y),\mu y^{-1})$ for some $\lambda,\mu\in \C^*$. As $\nu \circ \overline{\nu}=1$, we obtain  $\lambda\overline{\lambda}/\overline{\mu}=1$ and $\mu=\overline{\mu}$, whence $\mu\in \R_{>0}$. Considering $\alpha^{-1} \circ \nu \circ \overline{\alpha}$ with $\alpha=(x,\xi y)$, where $\xi\in \R^*$, $\xi^2=\mu$, we may reduce to the case where $\mu=1$, and consequently $\lvert \lambda \lvert=1$. Considering $\alpha^{-1} \circ \nu \circ \overline{\alpha}$ with $\alpha=(\varepsilon x,y)$, where $\varepsilon\in \C^*$ and $\varepsilon^2=\lambda$, we may further assume that $\lambda=1$. Then, the $1$-cocycle condition implies 
$\overline{c}(y)y+c(y^{-1})=0$. With $\alpha= (x- \overline{c}(y)y/2,y)$, we get $\alpha^{-1} \circ \nu \circ \overline{\alpha}= ( xy^{-1},y^{-1})$. Taking the morphism $\A^1_\C\times (\A^1_\C\setminus \{0\})\hookrightarrow \mathbb{P}^2, (x,y)\mapsto [x:y:1]$, we obtain the real structure $\rho'\colon [x:y:z]\mapsto [\overline{x}:\overline{z}:\overline{y}]$ on $\mathbb{P}^2 \setminus \{yz=0\}$. It remains to apply the automorphism $[x:y:z]\mapsto [y+z:\I (y-z):x]$ of $\mathbb{P}^2$, that gives an isomorphism $\mathbb{P}^2 \setminus \{yz=0\}\iso \mathbb{P}^2 \setminus \{x^2+y^2=0\}$, and conjugate the real structure $\rho'$ to the standard one. The corresponding real form is then isomorphic to $\mathbb{P}^2_{\R}\setminus \{x^2+y^2=0\}$.
\end{proof}

\section{Koras-Russell threefolds of the first kind}\label{Section:KR}
\subsection{Automorphisms of the three-space fixing the last coordinate}

Throughout this section, $\kk$ is a field and we denote by $x,y,z$ the coordinates of the affine three-space $\A^3_{\kk}=\Spec(\kk[x,y,z])$.

\begin{notation}
Let $\pi\colon \A^3_{\kk}\to \A^1_{\kk}$ be the projection $(x,y,z)\mapsto z$. Then, denote by $G_{\kk,z}$ the subgroup 
\begin{align*}G_{\kk,z}&=\{f\in \Aut(\A^3_{\kk})\mid \pi \circ f=\pi\}\\
&=\{f\in \Aut(\A^3_{\kk})\mid f^*(z)=z\}\\
&=\{f\in \Aut(\A^3_{\kk})\mid f=(P_1(x,y,z), P_2(x,y,z),z) \textrm{ with } P_1,P_2\in \kk[x,y,z]\}
\end{align*}
of all automorphisms of $\A^3_{\kk}$ that fix the last coordinate. 
\end{notation}

\medskip

Let $\kk\subseteq\KK$ be a field extension and let $f\in G_{\kk,z}$. Then, for each $q\in\KK$, we can define an automorphism $f|q$ of $\A^2_\KK=\Spec(\KK[x,y])$ by setting
\[f|q\colon(x,y)\mapsto(P_1(x,y,q), P_2(x,y,q)).\] 
We remark that $\Jac(f|q)=\Jac(f)\in\kk^*$.

\medskip

\begin{lemma}\label{Lemm:Psiq}Let $q\in\C\setminus\R$. Then, the map \[\begin{array}{rcc}
\Psi_q\colon G_{\R,z}&\to& \Aut(\A^2_\C),\\
f&\mapsto & f|q\end{array}\]
is a group homomorphism
whose image consists of all elements of $\Aut(\A^2_\C)$ that have a real Jacobian determinant.
\end{lemma}

\begin{proof}
By construction, $\Psi_q$ is a group homomorphism and $\Jac(\Psi_q(f))=\Jac(f)\in\R^*$ for all $f\in G_{\R,z}$. So, we only need to prove that every element $f$ in $\Aut(\A^2_\C)$ with real Jacobian determinant is indeed in the image of $\Psi_q$.

{\bf (A)} We prove that any element of $\Aff_2(\C)\cup\BA_2(\C)$ of Jacobian determinant one is in $\Psi_q(G_{\R,q})$.

Suppose first that $f$ is an elementary triangular map of the form $f=(x,y+\xi x^n)$ for some integer $n\ge 0$ and some constant $\xi\in \C$. Since $q$ is not real, there exist $s,t\in \R$ such that $sq+t=\xi$ and we then have $f=\Psi_q(g)$ where $g\in G_{\R,z}$ is defined by $g=(x,y+(sz+t)x^n,z)$. Since $\Psi_q$ is a group homomorphism, this implies that all triangular maps of the form $(x,y+p(x))$ with $p\in \C[x]$ also belong to $\Psi_q(G_{\R,q})$. 

We now consider affine maps. We have already proven that $(x,y+\lambda x)\in \Psi_q(G_{\R,q})$ for each $\lambda\in \C$.
Let us write $\sigma=(-y,x)=\Psi_q((-y,x,z)).$ As $\mathrm{SL}_2(\C)$ is generated by $\left(\begin{smallmatrix} 0 & -1 \\ 1 & 0 \end{smallmatrix}\right)$ and by $\left\{\left(\begin{smallmatrix} 1 & 0 \\ \lambda & 1 \end{smallmatrix}\right)\mid \lambda\in \C\right\}$, we can infer that every element $(ax+by,cx+dy)$ with $\left(\begin{smallmatrix} a & b \\ c & d \end{smallmatrix}\right)\in \mathrm{SL}_2(\C)$ belongs to $\Psi_q(G_{\R,q})$. As the translations are generated by $\sigma$ and by $(x,y+\nu)$ with $\nu \in \C$, every element of $\Aff_{2}(\C)$ of Jacobian determinant one lies in $\Psi_q(G_{\R,q})$. With the above, we can deduce that any element of $\BA_2(\C)$ of Jacobian determinant one is also in $\Psi_q(G_{\R,q})$, as it is of the form $(ax+b, \frac{1}{a}y+p(x))$ with $a \in \C^\ast$, $b\in \C$ and $p \in \C[x]$. This shows the claim.

{\bf (B)} Let $f\in\Aut(\A^2_\C)$ be such that $\Jac(f)\in\R^*$. By the Jung--van der Kulk theorem (Theorem~\ref{JungvdK}),  we can write $f$ as a product $\alpha_1\circ\alpha_2\circ\cdots\circ\alpha_{n+1}$, where $\alpha_1,\ldots,\alpha_{n+1}$ are elements of $\Aff_2(\C)\cup\BA_2(\C)$. We may change this expression by choosing some elements $\gamma_i=(a_ix,b_iy)\in\Aff_2(\C)\cap\BA_2(\C)$ and intercalating $\gamma_i\circ\gamma_i^{-1}$ between $\alpha_i$ and $\alpha_{i+1}$ as follows.
\begin{align*}
f&=\alpha_1\circ\alpha_2\circ\cdots\circ\alpha_i\circ\alpha_{i+1}\circ\cdots\circ\alpha_{n+1}\\
&=\alpha_1\circ(\gamma_1\circ\gamma_1^{-1})\circ\alpha_2\circ\cdots\circ\alpha_i\circ(\gamma_i\circ\gamma_i^{-1})\circ\alpha_{i+1}\circ\cdots\circ\alpha_{n+1}\\
&=(\alpha_1\circ\gamma_1)\circ(\gamma_1^{-1}\circ\alpha_2\circ\gamma_2)\circ\cdots\circ(\gamma_{i-1}^{-1}\circ\alpha_i\circ\gamma_i)\circ\cdots\circ(\gamma_n^{-1}\circ\alpha_{n+1})\\
&=\beta_1\circ\beta_2\circ\cdots\circ\beta_i\circ\cdots\circ\beta_{n+1},
\end{align*}
where $\beta_1=\alpha_1\circ\gamma_1$, $\beta_{n+1}=\gamma_n^{-1}\circ\alpha_{n+1}$ and $\beta_i=\gamma_{i-1}^{-1}\circ\alpha_i\circ\gamma_i$ for all $2\leq i\leq n$.

By choosing successively suitable elements $\gamma_1, \gamma_2,\ldots,\gamma_n$ as above, we can then arrange that $\Jac(\beta_{i})=1$ for all $1\leq i\leq n$. Hence, we can assume that $\beta_i\in\Psi_q(G_{\R,q})$ for all $i\leq n$ and it just remains to prove that $\beta_{n+1}$ belongs to $\Psi_q(G_{\R,q})$, too.

Since $\Jac(f)\in\R^*$ and $\Jac(\beta_i)=1$ for all $i\leq n$, we have that $\Jac(\beta_{n+1})=\Jac(f)\in\R^*$. Finally, we consider the map $\gamma_{n+1}=(\nu x,y)=\Psi_q((\nu x,y,z))$, where $\nu=\Jac(\beta_{n+1})$. This allows us to conclude the proof, because $\beta_{n+1}=(\beta_{n+1}\circ\gamma_{n+1}^{-1})\circ\gamma_{n+1}$ is in $\Psi_q(G_{\R,q})$ since $\beta_{n+1}\circ\gamma_{n+1}^{-1}$ is of Jacobian one.
\end{proof}

\begin{lemma}\label{Lem:vReal} The following propositions hold true.
\begin{enumerate}[leftmargin=*]
\item\label{vReal1}
Let $v\in \C[x,y]$ be a variable and let $q\in \C\setminus\R$. Then, there exists $f=(P_1(x,y,z),P_2(x,y,z),z)\in G_{\R,z}$ such that $v=P_1(x,y,q)$.
\item\label{vReal2}
Let $v\in \R[x,y]\subset\C[x,y]$ be a variable of $\C[x,y]$ and let $q\in\R$. Then, there exists $f=(P_1(x,y,z),P_2(x,y,z),z)\in G_{\R,z}$ such that $v=P_1(x,y,q)$.\end{enumerate}
\end{lemma}
\begin{proof}

\ref{vReal1} Suppose that $v\in \C[x,y]$ is a variable. Let $w\in \C[x,y]$ be such that $\varphi=(v,w)$ is an automorphism of $\A^2_\C$. Replacing $w$ with $\xi w$ for some $\xi\in \C^*$, we may assume that $\mathrm{Jac}(\varphi)=1$. Then, for any $q\in \C\setminus\R$, there exist by Lemma~\ref{Lemm:Psiq} polynomials $P_1,P_2\in \R[x,y,z]$ such that the automorphism 
  \[f=(P_1(x,y,z),P_2(x,y,z),z)\in G_{\R,z}\] satisfies $f|q=\varphi$. In particular, $v=P_1(x,y,q)$ as desired.
  
\ref{vReal2} It is a well-known fact that a polynomial $v\in\R[x,y]$ is a variable of $\R[x,y]$ if and only if it is a variable of $\C[x,y]$. For instance, this is an immediate consequence of \cite[Theorem 3.2]{Essen-Rossum}. Hence, there exists $w\in\R[x,y]$ such that $(v,w)\in\Aut(\A^2_{\R})$ and the result follows. Indeed, the map $f=(P_1(x,y,z),P_2(x,y,z),z)\in G_{\R,z}$ with $P_1(x,y,z)=v(x,y)$ and $P_2(x,y,z)=w(x,y)$ satisfies $v=P_1(x,y,q)$ for any $q\in\R$.
\end{proof}

\begin{proposition}\label{AutR} Consider the standard real structure $\rho\colon (x,y,z)\mapsto (\overline{x},\overline{y},\overline{z})$ on $\A^3_\C$.
Then, \[H^1(G_{\C,z})=\{1\}.\]
Consequently, every real structure $\hat\rho$ on $\A^3_\C$ that makes the following diagram commutative
\[
			\xymatrix@R=0.5cm@C=1.5cm{
			\A^3_\C\ar@{->}[r]^{\hat\rho}\ar@{->}[d]^{\pi}& \A^3_\C\ar@{->}[d]^{\pi}\\
			\A^1_\C \ar@{->}[r]^{z\mapsto \overline{z}}& \A^1_\C,
			}
			\]	
is equivalent to the standard real structure $\rho$.
\end{proposition}

\begin{proof}Let $\nu\in Z^1(G_{\C,z})$ be a $1$-cocycle, that is, an element $\nu\in G_{\C,z}$ such that $\nu\circ \overline{\nu}=\mathrm{id}_{\A^3_{\C}}$. We need to show that there exists $f\in G_{\C,z}$ such that $\nu=f^{-1}\circ\overline{f}$. 

Consider $G_{\C,z}$ as a subgroup of $\Aut(\A^2_K)$, where $K=\C(z)$ and $\A^2_K=\Spec(K[x,y])$. Since $H^1(\Aut_K(K[x,y]))=1$ by \cite[Theorem~3]{Kambayashi}, there is an element $f\in \Aut_{\C(z)}(\A^2_{\C(z)})$ such that $\nu=f^{-1}\circ \overline{f}$. In other words, there exist $f_1,f_2,g_1,g_2\in \C(z)[x,y]$ such that 
 \[\begin{array}{cccc}
f \colon &  \A^3_\C & \dasharrow &  \A^3_\C\\
&(x,y,z)&\mapsto & (f_1(x,y,z),f_2(x,y,z),z)
\end{array}\]
and 
\[\begin{array}{cccc}
g=f^{-1}\colon &  \A^3_\C & \dasharrow &  \A^3_\C\\
&(x,y,z)&\mapsto & (g_1(x,y,z),g_2(x,y,z),z),\\
\end{array}\]
are inverse birational maps and $\nu=f^{-1}\circ \overline{f}$.

We may actually assume that $f_1$ and $f_2$ are both elements of $\C[x,y,z]$. Indeed, there exists $c\in \R[z]\setminus \{0\}$ such that $c(z)f_1(x,y,z)$ and $c(z)f_2(x,y,z)$ belong to $\C[x,y,z]$, and the equality $\nu=f^{-1}\circ \overline{f}$ remains true when we replace $f$ with $\gamma\circ f$, where $\gamma\in \Bir(\A^3_\C)$ is defined by $(x,y,z)\mapsto (c(z)x,c(z)y,z)$, because $\overline{\gamma}=\gamma$.

Let us write $g_i=\frac{h_i}{a_i}$ for each $i=1,2$, where $h_i\in \C[x,y,z]$ and $a_i\in \C[z]\setminus \{0\}$ are without common factors. 

If $\deg(a_1\cdot a_2)=0$, i.e., if $a_1$ and $a_2$ are nonzero constants, then $g$ is a morphism too. In this case, $f$ is in $G_{\C,z}$ and we are done. If $\deg(a_1\cdot a_2)\geq1$, we proceed by decreasing induction on $\deg(a_1\cdot a_2)$. To prove the proposition, it suffices to find a suitable birational map $\varphi\in \Bir(\A^3_\C)$ with the following four properties:
\begin{enumerate}[leftmargin=*]
\item\label{pty1} $\varphi^*(z)=z$;
\item\label{pty2} all components of $\varphi\circ f$ are in $\C[x,y,z]$;
\item\label{pty3} $\varphi=\overline{\varphi}$, which implies $\nu=(\varphi\circ f)^{-1}\circ \overline{(\varphi\circ f)}$;
\item\label{pty4} the degree of the product of the denominators appearing in the components of $(\varphi\circ f)^{-1}$ is strictly smaller than that of $a_1\cdot a_2$. 
\end{enumerate}

\medskip
So, suppose from now on that $\deg(a_1 \cdot a_2)\geq1$ and let $q\in\C$ be such that $a_1(q)a_2(q)=0$. 
Without loss of generality, we may assume that $a_1(q)=0$. Since $g\circ f=\mathrm{id}_{\A^3_{\C}}$, we then obtain that 
\[a_1(z)x=h_1(f_1(x,y,z),f_2(x,y,z),z)\] and thus that the equality
\[\label{Eqhjf}h_1(f_1(x,y,q),f_2(x,y,q),q)=0\tag{$\spadesuit$}\]
holds in $\C[x,y]$.

For each $p\in \A^1_\C$, we consider the set $\Delta_{f,p}\subset \A^3_\C$ defined by \[\Delta_{f,p}=\{(f_1(x,y,p),f_2(x,y,p),p)\mid (x,y)\in \A^2_\C\}=f(\A^2_\C\times \{p\}).\]
We remark that applying the complex conjugation to the set $\Delta_{f,p}$ gives 
\[\label{EqDelta}\overline{\Delta_{f,p}}=\Delta_{f,\overline{p}}\tag{$\heartsuit$}\]
for all $p\in\C$. Indeed, as $\nu=f^{-1}\circ \overline{f}$, we have $f\circ \nu=\overline{f}=\rho\circ f \circ \rho$ and therefore
\[\overline{\Delta_{f,p}}=\rho(\Delta_{f,p})=\rho\circ f(\A^2_\C\times \{p\})=f\circ \nu\circ \rho(\A^2_\C\times \{p\})=f(\A_\C^2\times \{\overline{p}\})=\Delta_{f,\overline{p}}.\]

We shall prove later that if $\Delta_{f,q}$, with $a_1(q)=0$ as above, is not a point, then it isomorphic to $\A^1_{\C}$.

\medskip

{\bf (A)} Let us first consider the case where the set $\Delta_{f,q}$ is a point. Then, there exist $r_1,r_2\in\C$ such that $\Delta_{f,q}=\{(r_1,r_2,q)\}$ and $R_1,R_2\in\C[x,y,z]$ such that 
\[f_i(x,y,z)=r_i+(z-q)R_i(x,y,z)\]
for both $i=1, 2$.

{\bf (A1)} Suppose now that $q\in \R$. By the equality \eqref{EqDelta}, we then have that $r_1,r_2\in \R$. 
Therefore, the birational map 
\[\varphi=\left(\frac{x-r_1}{z-q}, \frac{y-r_2}{z-q}, z\right)\in \Bir(\A^3_\C)\] 
satisfies $\varphi=\overline{\varphi}$ and we compute  
\[\varphi\circ f=(R_1(x,y,z), R_2(x,y,z), z).\]
The inverse map of $\varphi\circ f$ is given by 
\begin{align*}(\varphi\circ f)^{-1}&=\left(\frac{h_1(x(z-q)+r_1,y(z-q)+r_2,z)}{a_1(z)}, \right.\\
&\qquad\left.\frac{h_2(x(z-q)+r_1,y(z-q)+r_2,z)}{a_2(z)}, z\right)\\
&=\left(\frac{\widetilde{h}_1(x,y,z)}{a_1(z)}, \frac{\widetilde{h}_2(x,y,z)}{a_2(z)}, z\right),
\end{align*}
where $\widetilde{h}_1,\widetilde{h}_2\in\C[x,y,z]$. We obtain that
\[\widetilde{h}_1(x,y,q)=h_1(r_1,r_2,q)=h_1(f_1(x,y,q),f_2(x,y,q),q)\stackrel{\eqref{Eqhjf}}{=}0.\] 
Therefore, $\widetilde{h}_1(x,y,z)$ is divisible by $(z-q)$ and the map $\varphi$ fulfils the four desired properties \ref{pty1}--\ref{pty4}.

{\bf (A2)} We now consider the case where $q\not\in \R$. For each $i=1,2$, we define two real numbers $s_i=\frac{\overline{r_i}-r_i}{\overline{q}-q}$ and $t_i=\frac{\overline{q}r_i-q\overline{r_i}}{\overline{q}-q}$. Then, the polynomials $p_i(z)=s_iz+t_i\in\R[z]$ satisfy that $p_i(q)=r_i$ and $p_i(\overline{q})=\overline{r_i}$. We recall that the equality $f_i(x,y,q)=r_i$ holds true in $\C[x,y]$. Similarly, it follows from~\eqref{EqDelta}, that $f_i(x,y,\overline{q})=\overline{r_i}$. Therefore the polynomials $f_i(x,y,z)-s_iz-t_i$ are divisible by $(z-q)(z-\overline{q})\in\R[z]$. 
This implies that the birational map 
\[\varphi=\left(\frac{x-s_1z-t_1}{(z-q)(z-\overline{q})}, \frac{y-s_2z-t_2}{(z-q)(z-\overline{q})}, z\right)\] 
satisfies $\varphi=\overline{\varphi}$ and that all components of $\varphi\circ f$ are elements of $\C[x,y,z]$.
Moreover, the inverse map of $\varphi\circ f$ is then given by 
\begin{align*}(\varphi\circ f)^{-1}&=\left(\frac{h_1(x(z-q)(z-\overline{q})+s_1z+t_1,y(z-q)(z-\overline{q})+s_2z+t_2,z)}{a_1(z)},\right.\\ 
&\quad\left.\frac{h_2(x(z-q)(z-\overline{q})+s_1z+t_1, y(z-q)(z-\overline{q})+s_2z+t_2,z)}{a_2(z)}, z\right)\\
&=\left(\frac{\widetilde{h}_1(x,y,z)}{a_1(z)}, \frac{\widetilde{h}_2(x,y,z)}{a_2(z)}, z\right),
\end{align*}
where $\widetilde{h}_1,\widetilde{h}_2\in\C[x,y,z]$. We obtain the two equalities
\[\widetilde{h}_1(x,y,q)=h_1(r_1,r_2,q)=h_1(f_1(x,y,q),f_2(x,y,q),q)\stackrel{\eqref{Eqhjf}}{=}0,\]
\[\widetilde{h}_1(x,y,\overline{q})=h_1(\overline{r_1},\overline{r_2},\overline{q})=h_1(f_1(x,y,q),f_2(x,y,q),q)\stackrel{\eqref{Eqhjf}}{=}0.\]

Therefore, $\widetilde{h}_1(x,y,z)$ is divisible by $(z-q)(z-\overline{q})$ and the map $\varphi$ fulfils the four desired properties \ref{pty1}--\ref{pty4}.

\medskip
{\bf (B)} We now proceed with the case where $\Delta_{f,q}$ is not a point. For every $p\in \A^1_\C$ and every variable $u\in \C[x,y]$, we define the curve 
\[\Gamma_{p,u}=\{(x,y,p)\in \A^3_\C\mid u(x,y)=0\}\simeq \A^1_\C.\]

Consider the polynomial $h_1(x,y,q)\in\C[x,y]$. It is an element of the closure of the set of all variables of $\C[x,y]$  in the ind-topology (see \cite{Furter}). By \cite{Furter} (see also \cite[Corollary 16.7.5]{Furter-Kraft}), there exists a variable $v\in \C[x,y]$ such that $h_1(x,y,q)\in \C[v]$. Note that $h_1(x,y,q)$ is not the zero-polynomial because $h_1$ and $a_1$ were chosen without common factors. Setting $\mu\in\C^*$, $m\geq1$ and $\xi_1,\ldots,\xi_m\in\C$ such that 
\[h_1(x,y,q)=\mu\prod_{i=1}^m(v(x,y)-\xi_i)\in\C[x,y],\]
it then follows from \eqref{Eqhjf} that there exists $1\leq i \leq m$ such that the equality 
\[v(f_1(x,y,q),f_2(x,y,q))-\xi_i=0\] 
holds true in $\C[x,y]$. Therefore, the set $\Delta_{f,q}$ is contained in the curve $\Gamma_{q,w}$, where $w=v-\xi_i$. As a nonconstant morphism $\A^2_\C\to \A^1_\C$ is surjective, and since $\Delta_{f,q}$ is not a point, this implies that \[\Delta_{f,q}=\Gamma_{q,w}\simeq\A^1_{\C}.\]
\par We now prove that we can assume that $w\in \R[x,y]$ if $q\in \R$, so that we may apply Lemma~\ref{Lem:vReal}. Indeed, suppose $q\in\R$. In this case, we have $\overline{\Delta_{f,q}}\stackrel{\eqref{EqDelta}}{=}\Delta_{f,\overline{q}}=\Delta_{f,q}$. Thus, $\Gamma_{q,w}=\Gamma_{q,\overline{w}}$. As both polynomials $w$ and $\overline{w}$ are variables, they are irreducible. Since their zero-sets are equal, there exists a constant $\mu\in \C^*$ such that $\overline{w}=\mu w$. It then follows that $w=\overline{\mu}\,\overline{w}=\overline{\mu} \, \overline{w}=\mu \overline{\mu}w$, whence $\mu \overline{\mu}=1$. As $H^1(\C^*)=\{1\}$ by Lemma~\ref{Lem:H1CCstar}, we may choose   $\eta\in \C^*$ with $\eta/\overline{\eta}=\mu$. The variable $w'=\eta w$ then satisfies $\overline{w'}=\overline{\eta}\,\overline{w}=\overline{\eta}\mu w=w'$ and $\Delta_{f,q}=\Gamma_{q,w'}$. Thus, we may replace $w$ by $w'\in\R[x,y]$ if necessary, as desired.

By Lemma~\ref{Lem:vReal}, there exists an element $\psi=(P_1(x,y,z), P_2(x,y,z), z)$ in $G_{\R,z}$ such that $P_1(x,y,q)=w$. Observe that $\psi(\Gamma_{q,w})\subseteq \Gamma_{q,x}$. As these two curves are isomorphic to $\A^1_\C$, they are actually equal, i.e., $ \psi(\Gamma_{q,w})= \Gamma_{q,x}$. We may thus replace $f$ with $ \psi\circ f$ and suppose that $\Delta_{f,q}=\Gamma_{q,x}$. Note that, as $\psi\in G_{\R,z}$ is defined over $\R$ and is an automorphism, the equality $\nu=f^{-1}\circ \overline{f}$ is preserved when replacing $f$ with $\psi\circ f$, and we do not change the denominators $a_1,a_2$ appearing in the expression of the inverse of $f$. Moreover, the fact that $\Delta_{f,q}=\Gamma_{q,x}$ implies that $f_1(x,y,q)=0$, or, equivalently, that $z-q$ divides $f_1$ in $\C[x,y,z]$. We note that in the case where $q\not\in \R$, we also have $\Delta_{f,\overline{q}}=\overline{\Delta_{f,q}}=\Gamma_{\overline{q},x}$, and so $(z-\overline{q})$ also divides $f_1$. Defining $u(z)=z-q$ if $q\in \R$ and $u(z)=(z-q)(z-\overline{q})$ if $q\not\in \R$, we thus get a polynomial $u\in \R[z]$ with $u(q)=0$ that divides $f_1$.

Finally, since the birational map $\varphi\in \Bir(\A^3_\C)$ defined by $\varphi\colon (x,y,z)\mapsto (\frac{x}{u(z)},y,z)$ satisfies the four properties \ref{pty1}--\ref{pty4}, we can conclude the proof.
\end{proof}

\begin{corollary}\label{coro:AutCzxy}
Taking the standard action of $\mathrm{Gal}(\C/\R)$ on  $\C[x,y,z]$, we obtain 
\[H^1(\Aut_{\C[z]}(\C[x,y,z]))=\{1\}.\]
\end{corollary}
\begin{proof}
The map $f\mapsto (f^{-1})^*$ defines an isomorphism between the groups $G_{\C,z}$ and $\Aut_{\C[z]}(\C[x,y,z])$. As the action of $\mathrm{Gal}(\C/\R)$ on both groups is compatible with this isomorphism, $H^1(\Aut_{\C[z]}(\C[x,y,z]))=\{1\}$ then follows from $H^1(G_{\C,z})=\{1\}$. 
\end{proof}
\begin{lemma}\label{Lemm:GrH1}
Let $r\ge 1$ and let 
\[G_r=\{f\in \Aut_{\C[z]}(\C[x,y,z])\mid f\equiv \mathrm{id} \pmod{z^r}\}.\]
Taking the standard action of $\mathrm{Gal}(\C/\R)$ on  $\C[x,y,z]$, we obtain 
\[H^1(G_r)=\{1\}.\]
\end{lemma}
\begin{proof}
{\bf (A)} 
We first prove the result in the case where $r=1$. Let $\nu\in Z^1(G_1)$. By Corollary~\ref{coro:AutCzxy}, there exists $\alpha\in\Aut_{\C[z]}(\C[x,y,z])$ such that $\nu=\alpha^{-1}\circ\overline{\alpha}$. Since $\alpha\circ\nu=\overline{\alpha}$ and since $\nu\equiv\mathrm{id}$ $\pmod{z}$, we have that $\alpha\equiv\overline{\alpha}\pmod{z}$.

Denoting $\alpha(x)=a(x,y,z)$ and $\alpha(y)=b(x,y,z)$, we can define an automorphism $\varphi\in\Aut_{\C[z]}(\C[x,y,z])$ by letting $\varphi(x)=a(x,y,0)$ and $\varphi(y)=b(x,y,0)$. Note that $\varphi\equiv\alpha\pmod{z}$ and that $\overline{\varphi}=\varphi$.

Thus, $\beta=\varphi^{-1}\circ\alpha$ defines an element in $G_1$ and we check that
\[\beta^{-1}\circ \overline{\beta}=\alpha^{-1}\circ\varphi\circ\overline{\varphi^{-1}}\circ\overline{\alpha}=\alpha^{-1}\circ\varphi\circ\varphi^{-1}\circ\overline{\alpha}=\alpha^{-1}\circ\overline{\alpha}=\nu.\] Hence, $H^1(G_1)=\{1\}$ is proved.

{\bf (B)} We prove the lemma for every $r\geq2$ by induction on $r$. We fix $r\geq2$ and suppose that $H^1(G_{r-1})=\{1\}$ holds. Let $\nu\in Z^1(G_r)$. We want to find an element $\beta\in G_{r}$ such that $\nu=\beta^{-1}\circ \overline{\beta}$. By our induction hypothesis, there exists $\alpha\in G_{r-1}$ such that $\nu=\alpha^{-1}\circ \overline{\alpha}$.

 Now, it suffices to construct an element $\varphi\in\Aut_{\C[z]}(\C[x,y,z])$ with $\varphi=\overline{\varphi}$ such that $\varphi\equiv\alpha\pmod{z^r}$. The lemma will indeed follow since the automorphism $\beta=\varphi^{-1}\circ\alpha$ is then in $G_r$ and satisfies 
\[\beta^{-1}\circ \overline{\beta}=\alpha^{-1}\circ\varphi\circ\overline{\varphi^{-1}}\circ\overline{\alpha}=\alpha^{-1}\circ\varphi\circ\varphi^{-1}\circ\overline{\alpha}=\alpha^{-1}\circ\overline{\alpha}=\nu,\] as desired.

Let $a,b\in\mathbb{C}[x,y]$ be such that $\alpha(x)\equiv x+z^{r-1}a(x,y)$ and $\alpha(y)\equiv y+z^{r-1}b(x,y)\pmod{z^r}$. Since $\alpha\circ\nu=\overline{\alpha}$ and $\nu\equiv\mathrm{id} \pmod{z^r}$, we have that $a(x,y)$ and $b(x,y)$ both belong to $\mathbb{R}[x,y]$. Therefore, $\alpha$ induces an endomorphism $\tilde{\alpha}\in \End_{\R[z]/(z^r)}(\R[z]/(z^r)[x,y])$ defined by $\tilde{\alpha}(x)=x+z^{r-1}a(x,y)$ and $\tilde{\alpha}(y)= y+z^{r-1}b(x,y)$. In fact, $\tilde{\alpha}$ is an isomorphism. Indeed, one can check that its inverse map is simply defined by $\tilde{\alpha}^{-1}(x)=x-z^{r-1}a(x,y)$ and $\tilde{\alpha}^{-1}(y)= y-z^{r-1}b(x,y)$. Moreover, the Jacobian determinant of $\tilde{\alpha}$ is equal to $1\in\R[z]/(z^r)$ because  $\tilde\alpha\equiv\alpha\pmod{z^r}$ and  $\Jac(\alpha)=1\in \R[z]$.

 By the main result of \cite{EMV}, there thus exists $\varphi\in\Aut_{\R[z]}(\R[z][x,y])$ with $\varphi(x)\equiv x+z^{r-1}a(x,y)\equiv \alpha(x)$ and $\varphi(y)\equiv y+z^{r-1}b(x,y)\equiv \alpha(y)\pmod{z^r}$. This concludes the proof.
\end{proof}

\subsection{Real forms of Koras-Russell threefolds of the first kind}\label{section:KR}

The Koras-Russell threefolds of the first kind are the hypersurfaces $X_{d,k,\ell}$ in $\A^4_{\C}$ defined by an equation of the form $x^dy+z^k+x+t^\ell=0$, where $d\geq2$ and $2\leq k<\ell$ are integers with $k$ and $\ell$ relatively prime. Their automorphism groups are computed in \cite{DMJP, Moser-Jauslin}, see also \cite{DMJP-proceedings} where the following notations are introduced. We fix the integers $d,k,\ell$ as above, and denote by $\mathcal{A}\subset \Aut_{\C}(\C[x,z,t])$ the subgroup of all automorphisms of $\C[x,z,t]$ that preserves the ideals $(x)$ and $(x^d,z^k+x+t^\ell)$. For every $1\leq n\leq d$ we further denote  by $\mathcal{A}_n$ the normal subgroup of $\mathcal{A}$ defined by
\[\mathcal{A}_n=\{f\in\mathcal{A} \mid f\equiv \id\mod(x^n)\}\subset\Aut_{\C[x]}(\C[x,z,t]).\]

\begin{proposition}[\cite{Moser-Jauslin}]\label{prop:Automorphisms_KR}\hfill
\begin{enumerate}[leftmargin=*]
\item $\Aut_\C(X_{d,k,\ell}) \simeq\mathcal{A}$. \label{item: auto of threeforld}
\item $\mathcal{A} \simeq\mathcal{A}_1\rtimes \C^*$, where $\C^*$ acts on $\C[x,z,t]$ by $a\cdot P(x,z,t)=P(a^{k\ell}x,a^{\ell}z,a^{k}t)$ for all $a\in\C^*, P\in\C[x,z,t]$. \label{item: semidirect product}
\item $\mathcal{A}_{n+1}$ is a normal subgroup of $\mathcal{A}_{n}$ and $\mathcal{A}_{n}/\mathcal{A}_{n+1}\simeq(\C[z,t],+)$ for all $1\leq n\leq d-1$. \label{item: normal subgroups}
\end{enumerate}
\end{proposition}
Moreover, the above isomorphisms are compatible with the natural action of $\Gal(\C/\R)$ on $\C^*$ and on the polynomial rings $\C[z,t]\subset\C[x,z,t]\subset\C[x,x^{-1},y,z,t]$, where we see the ring $\C[X_{d,k,\ell}]$ of regular functions on $X_{d,k,\ell}$ as the subalgebra of $\C[x,x^{-1},z,t]$ that is generated by $x$, $z$, $t$ and $y=-(z^k+x+t^\ell)/x^d$. We may now prove Theorem~\ref{ThmB}.

\begin{proof}[Proof of Theorem~\ref{ThmB}]
By Proposition~\ref{prop:Automorphisms_KR}\ref{item: normal subgroups}, we have a subnormal series \[\{1\} \triangleleft\mathcal{A}_d \triangleleft \mathcal{A}_{d-1} \triangleleft \cdots \triangleleft \mathcal{A}_1 \triangleleft \mathcal{A},\]
where $\mathcal{A}_n/\mathcal{A}_{n+1} \simeq (\C[z,t],+)$ for each $1\leq n \leq d-1$. We may write the latter isomorphism in the form of a short exact sequence of group homomorphisms
\[ \{1\}\longrightarrow \mathcal{A}_{n+1}\longrightarrow \mathcal{A}_{n}\longrightarrow (\C[z,t],+) \longrightarrow \{1\},\]
that gives rise to a short exact sequence of homomorphisms of pointed sets 
\[ \{1\}\longrightarrow H^1(\mathcal{A}_{n+1})\longrightarrow H^1(\mathcal{A}_{n})\longrightarrow H^1(\C[z,t]) \longrightarrow \{1\}\]
(see for example \cite[Proposition 1.17.]{Borel-Serre}). Observe that 
\[\mathcal{A}_{d}=\{f\in\Aut_{\C[x]}(\C[x,z,t])\mid f\equiv\mathrm{id}\mod (x^d)\},\] for which Lemma~\ref{Lemm:GrH1} implies that the first cohomology pointed set $H^1(\mathcal{A}_d)$ is trivial. Since, by Lemma~\ref{Lem:H1CCstar}, $H^1(\C[z,t])$ is trivial, $H^1(\mathcal{A}_{d-1})$ is, too. By repeating the same argument, we  see that the triviality of $H^1(\mathcal{A}_{n+1})$ implies that of $H^1(\mathcal{A}_{n})$. Hence, we successively find that all cohomology pointed sets $H^1(\mathcal{A}_d),\ldots, H^1(\mathcal{A}_1)$ are trivial.

Now, by Proposition~\ref{prop:Automorphisms_KR}\ref{item: semidirect product}, we again obtain a short exact sequence of group homomorphisms
\[ \{1\}\longrightarrow \mathcal{A}_{1}\longrightarrow \mathcal{A}\longrightarrow \C^* \longrightarrow \{1\},\] and thus a short exact sequence of homomorphisms of pointed sets 
\[ \{1\}\longrightarrow H^1(\mathcal{A}_{1})\longrightarrow H^1(\mathcal{A})\longrightarrow H^1(\C^*) \longrightarrow \{1\}.\]
As $H^1(\mathcal{A}_1)=\{1\}$ by the preceding argument, and since $H^1(\C^*)$ is trivial by Lemma~\ref{Lem:H1CCstar}, we can deduce that $H^1(\mathcal{A})$ is trivial. 

Therefore, $H^1(\Aut_{\C}(X_{d,k,\ell}))$ is also trivial and we obtain that all real forms of $X_{d,k,\ell}$ are isomorphic to the standard one.
\end{proof}

\bibliographystyle{alpha}
\bibliography{biblio}

\end{document}